\documentclass[12pt,reqno]{amsart}
\usepackage{amsfonts,amssymb}
\usepackage{color}
\usepackage{graphicx}
\usepackage{dsfont}
\usepackage{bbm}

\usepackage{a4wide}

\makeatletter
\def\eqnarray{\stepcounter{equation}\let\@currentlabel=\theequation
\global\@eqnswtrue
\tabskip\@centering\let\\=\@eqncr
$$\halign to \displaywidth\bgroup\hfil\global\@eqcnt\z@
  $\displaystyle\tabskip\z@{##}$&\global\@eqcnt\@ne
  \hfil$\displaystyle{{}##{}}$\hfil
  &\global\@eqcnt\tw@ $\displaystyle{##}$\hfil
  \tabskip\@centering&\llap{##}\tabskip\z@\cr}

\def\endeqnarray{\@@eqncr\egroup
      \global\advance\c@equation\m@ne$$\global\@ignoretrue}

\def\@yeqncr{\@ifnextchar [{\@xeqncr}{\@xeqncr[5pt]}}
\makeatother

\newtheorem{theorem}{Theorem}[section]
\newtheorem{lemma}[theorem]{Lemma}
\newtheorem{corollary}[theorem]{Corollary}
\newtheorem{proposition}[theorem]{Proposition}

\theoremstyle{definition}

\newtheorem{definition}[theorem]{Definition}
\newtheorem{assu}[theorem]{Assumption}

\newtheorem{rem}[theorem]{Remark}

\theoremstyle{remark}

\numberwithin{equation}{section}

\newcommand{\SX}{\Sigma}

\newcommand{\field}[1]{\mathbb{#1}}
\newcommand{\N}{\field{N}}
\newcommand{\R}{\field{R}}
\newcommand{\C}{\field{C}}

\DeclareMathOperator{\tr}{tr}      
\DeclareMathOperator{\dom}{Dom}    
\DeclareMathOperator{\supp}{supp}  
\DeclareMathOperator{\sgn}{sgn}  

\newcommand{\norm}[1]{\lVert#1\rVert}

\newcommand{\upp}[1]{{#1}^\bullet}  
\newcommand{\low}[1]{{#1}_\bullet}

\newcommand{\RRe}{\mathop{\rm Re}}
\newcommand{\IIm}{\mathop{\rm Im}}
\newcommand{\loc}{{\rm loc}}
\newcommand{\one}{\mathds{1}}

\newcommand{\cb}{{\mathcal{B}}}
\newcommand{\cd}{{\mathcal{D}}}

\newcommand{\gt}{\mathfrak{t}} 
\newcommand{\gE}{\mathfrak{E}}   
   
\newcommand{\gJ}{\mathfrak{J}}  

\newcommand{\LL}{\mathbb{L}}

\newcounter{teller}

\newenvironment{statements}{\begin{list}%
{\rm  (\roman{teller})\hfill}{\usecounter{teller} \leftmargin=1.1cm
\labelwidth=1.1cm \labelsep=0cm \parsep=0cm}
                           }{\end{list}}

\def\typeout#1{\message{^^J}\message{#1}\message{^^J}}
\newif\ifSRCOK \SRCOKtrue
\newcount\PAGETOP
\newcount\LASTLINE
\global\PAGETOP=1
\global\LASTLINE=-1
\def\EJECT{\SRC\eject}
\def\WinEdt#1{\typeout{:#1}}
\gdef\MainFile{\jobname.tex}
\gdef\CurrentInput{\MainFile}
\newcount\INPSP
\global\INPSP=0
\def\SRC{\ifSRCOK%
  \ifnum\inputlineno>\LASTLINE%
    \ifnum\LASTLINE<0%
      \global\PAGETOP=\inputlineno%
    \fi%
    \global\LASTLINE=\inputlineno%
    \ifnum\INPSP=0%
      \ifnum\inputlineno>\PAGETOP%
        
      \fi%
    \else%
      
    \fi%
  \fi%
\fi}
\def\PUSH#1{%
\SRC%
\ifnum\INPSP=0 \global\let\INPSTACKA=\CurrentInput \else%
\ifnum\INPSP=1 \global\let\INPSTACKB=\CurrentInput \else%
\ifnum\INPSP=2 \global\let\INPSTACKC=\CurrentInput \else%
\ifnum\INPSP=3 \global\let\INPSTACKD=\CurrentInput \else%
\ifnum\INPSP=4 \global\let\INPSTACKE=\CurrentInput \else%
\ifnum\INPSP=5 \global\let\INPSTACKF=\CurrentInput \else%
               \global\let\INPSTACKX=\CurrentInput \fi\fi\fi\fi\fi\fi%
\gdef\CurrentInput{#1}%
\WinEdt{<+ \CurrentInput}%
\global\LASTLINE=0%
\ifSRCOK\fi%
\global\advance\INPSP by 1}
\def\POP{%
\ifnum\INPSP>0 \global\advance\INPSP by -1  \fi%
\ifnum\INPSP=0 \global\let\CurrentInput=\INPSTACKA \else%
\ifnum\INPSP=1 \global\let\CurrentInput=\INPSTACKB \else%
\ifnum\INPSP=2 \global\let\CurrentInput=\INPSTACKC \else%
\ifnum\INPSP=3 \global\let\CurrentInput=\INPSTACKD \else%
\ifnum\INPSP=4 \global\let\CurrentInput=\INPSTACKE \else%
\ifnum\INPSP=5 \global\let\CurrentInput=\INPSTACKF \else%
               \global\let\CurrentInput=\INPSTACKX \fi\fi\fi\fi\fi\fi%
\WinEdt{<-}%
\global\LASTLINE=\inputlineno%
\global\advance\LASTLINE by -1%
\SRC}
\def\INPUT#1{\relax}

\let\originalxxxeverypar\everypar
\newtoks\everypar
\originalxxxeverypar{\the\everypar\expandafter\SRC}
\everymath\expandafter{\the\everymath\expandafter\SRC}
\output\expandafter{\expandafter\SRCOKfalse\the\output}

\newif\ifSRCOK \SRCOKtrue
\newcount\PAGETOP
\newcount\LASTLINE
\global\PAGETOP=1
\global\LASTLINE=-1
\gdef\MainFile{\jobname.tex}
\gdef\CurrentInput{\MainFile}
\newcount\INPSP
\global\INPSP=0
\def\EJECT{\SRC\eject}
\def\WinEdt#1{\typeout{:#1}}
\def\SRC{\ifSRCOK%
  \ifnum\inputlineno>\LASTLINE%
    \ifnum\LASTLINE<0%
      \global\PAGETOP=\inputlineno%
    \fi%
    \global\LASTLINE=\inputlineno%
    \ifnum\INPSP=0%
      \ifnum\inputlineno>\PAGETOP%
      \fi%
    \else%
    \fi%
  \fi%
\fi}
\def\PUSH#1{%
\SRC%
\ifnum\INPSP=0 \global\let\INPSTACKA=\CurrentInput \else%
\ifnum\INPSP=1 \global\let\INPSTACKB=\CurrentInput \else%
\ifnum\INPSP=2 \global\let\INPSTACKC=\CurrentInput \else%
\ifnum\INPSP=3 \global\let\INPSTACKD=\CurrentInput \else%
\ifnum\INPSP=4 \global\let\INPSTACKE=\CurrentInput \else%
\ifnum\INPSP=5 \global\let\INPSTACKF=\CurrentInput \else%
               \global\let\INPSTACKX=\CurrentInput \fi\fi\fi\fi\fi\fi%
\gdef\CurrentInput{#1}%
\WinEdt{<+ \CurrentInput}%
\global\LASTLINE=0%
\ifSRCOK\fi%
\global\advance\INPSP by 1}
\def\POP{%
\ifnum\INPSP>0 \global\advance\INPSP by -1  \fi%
\ifnum\INPSP=0 \global\let\CurrentInput=\INPSTACKA \else%
\ifnum\INPSP=1 \global\let\CurrentInput=\INPSTACKB \else%
\ifnum\INPSP=2 \global\let\CurrentInput=\INPSTACKC \else%
\ifnum\INPSP=3 \global\let\CurrentInput=\INPSTACKD \else%
\ifnum\INPSP=4 \global\let\CurrentInput=\INPSTACKE \else%
\ifnum\INPSP=5 \global\let\CurrentInput=\INPSTACKF \else%
               \global\let\CurrentInput=\INPSTACKX \fi\fi\fi\fi\fi\fi%
\WinEdt{<-}%
\global\LASTLINE=\inputlineno%
\global\advance\LASTLINE by -1%
\SRC}
\def\INPUT#1{\relax}
\let\OldINCLUDE=\include
\def\include#1{
\EJECT%
\PUSH{#1.tex}%
\OldINCLUDE{#1}%
\POP}

\let\originalxxxeverypar\everypar
\newtoks\everypar
\originalxxxeverypar{\the\everypar\expandafter\SRC}
\everymath\expandafter{\the\everymath\expandafter\SRC}
\let\zzzxxxbibliography=\bibliography
\def\bibliography#1{\PUSH{\jobname.bbl}\zzzxxxbibliography{#1}\POP}
\output\expandafter{\expandafter\SRCOKfalse\the\output}

\begin{document}
\bibliographystyle{tom}

\author{A.F.M. ter Elst}
\address{Department of Mathematics, The University of Auckland, Private Bag 92019,  
Auckland 1142, New Zealand.}
\email{terelst@math.auckland.ac.nz}

\author{M. Meyries}
\address{Department of Mathematics,
Karlsruhe Institute of Technology, 76128 Karlsruhe, Germany.}
\email{martin.meyries@kit.edu}

\author{J. Rehberg}
\address{Weierstrass Institute, Mohrenstr.\ 39, 10117 Berlin, Germany.}
\email{joachim.rehberg@wias-berlin.de}

\title[Parabolic equations with source terms on interfaces]{Parabolic equations 
with dynamical boundary conditions and source terms on interfaces}

\keywords{parabolic equation, quasilinear parabolic problem, mixed boundary condition, dynamical boundary condition, maximal parabolic $L^p$-regularity, nonsmooth geometry, nonsmooth coefficients}
\subjclass[2000]{35K20, 35K59, 35M13, 35R05}
\thanks{Part of this work is supported by the Marsden Fund Council from Government funding, 
administered by the Royal Society of New Zealand, and by the Deutsche Forschungsgemeinschaft (DFG project ME 3848/1-1).}

\date{\today}
\begin{abstract} We consider parabolic equations with mixed boundary conditions 
and domain inhomogeneities supported on a lower dimensional hypersurface, 
enforcing a jump in the conormal derivative.
Only minimal regularity assumptions on the domain and the coefficients are imposed.
It is shown that the corresponding linear operator enjoys maximal parabolic regularity in a suitable 
$L^p$-setting.
The linear results suffice
to treat also the corresponding nondegenerate quasilinear problems.
\end{abstract}
\renewcommand{\MakeUppercase}[1]{{\sc #1}}

\maketitle

\section{Introduction}\label{s-intro}
In this article we are interested in the linear parabolic initial-boundary value problem
\begin{alignat}{3}
\varepsilon \partial_t u-\nabla \cdot \mu \nabla u & =  f_\Omega  & 
\qquad & \text{in }\,J \times (\Omega \setminus \Sigma) , \label{e-parabol}\\
u & =  0  & &\text {on }\,  J \times (\partial \Omega \setminus \Gamma), \label{e-Diri}\\
\varepsilon \partial_t u+\nu \cdot \mu \nabla u + b u & =  f_\Gamma &  & 
\text{on }\, J \times \Gamma, \label{e-robin}\\
\varepsilon \partial_t u+[\nu_\SX \cdot \mu \nabla u]  & =   f_\SX 
  && \text{on }\, J \times \SX, \label{e-xi-eq}\\
u(0) & =  u_0  & &\text{in }\, \Omega \cup \Gamma, \label{e-initial}
\end{alignat}
and in its quasilinear variants.
Here $J= (0,T)$ is a bounded time interval,
 $\Omega \subset \R^d$ is a bounded domain, $\Gamma\subseteq \partial\Omega$ is a part of the
 boundary with outer normal~$\nu$, and  $\SX\subset \Omega$ is e.g.\ a finite union of hypersurfaces,
 equipped with a normal field $\nu_\SX$.
By $[\nu_\SX \cdot \mu \nabla u]$ we denote the
jump of $\nu_\SX \cdot \mu \nabla u$ over $\SX$.
 The case that $\Gamma$ or $\SX$ is an empty set
 is not excluded.
We treat a nonsmooth geometry; e.g., it suffices that $\Gamma$ and $\SX$ satisfy
 certain Lipschitz conditions.
Nothing is supposed on the Dirichlet part 
$\partial\Omega \setminus \Gamma$ of the boundary, and the boundary parts $\Gamma$ and 
$\partial\Omega \setminus \Gamma$ are allowed to meet.

Also on the coefficients we impose only low regularity conditions.
The (possibly nonsymmetric)
 coefficient matrix $\mu$ is  bounded and uniformly elliptic, $\varepsilon$ is  positive, 
bounded and bounded away from zero, and  $b$ only has to live in an $L^p$-space.
The 
(possibly nonautonomous) inhomogeneities $f_\Omega$, $f_\Gamma$, $f_{\SX}$ and the 
initial value $u_0$ are assumed to be given.

Parabolic problems with dynamical boundary conditions are considered by many authors,
 see e.g.\ \cite{AmE}, \cite{Esc2}, \cite{Hin} \cite{AQR}, \cite{BBR} and
\cite{BD}, but there always severe assumptions on the data, as smoothness, are
 imposed (compare also \cite{FGGR} and \cite{VV},
 where the boundary condition on $J \times \Gamma$ is understood as Wentzell's boundary condition). 
It is the aim of this work to show that any smoothness assumption on the domain
and the coefficient function $\mu$ can be avoided.
In particular, the domain $\Omega$ does not need to be a Lipschitz domain. 
Let us briefly comment
on this: a moment's thought shows that, by far, many natural domains fail to be 
Lipschitzian. 
For example, if one removes from a ball one half of its equatorial plane,
then the remainder fails to be Lipschitzian. 
As another example, consider a pair of pincers as in Figure~1. 
It is also not Lipschitzian. 
\begin{figure}[htbp]
\centerline{\includegraphics[scale=0.7]{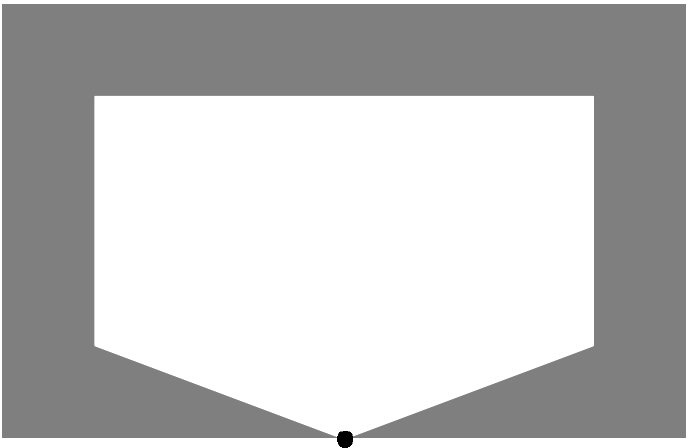}}
\caption{\label{fig-gegenbeispiel} A pair of pincers is \emph{not} a Lipschitz domain.}
\end{figure}
The crucial point is that such objects, obviously, occur in the 
physical world.
In this paper we also allow the
 inhomogeneities not only to
 live in the volume of the domain, but to incorporate a part which is supported on
 the set $\SX$ of lower dimension $d-1$.
This largely extends the applicability of
 the theory to real-world problems.
The reader may think,
 e.g., of a heat source which is concentrated on an interface.
Alternatively, one
 meets such 
constellations in electricity: surface charge densities induce a jump in the
normal component of the dielectric displacement, see e.g.\ \cite[Chapter~1]{Tam}. 

Our approach to (\ref{e-parabol})--(\ref{e-initial}),  which also covers
the case that $\Gamma$ or $\SX$ is empty, is essentially based on the theory of 
sesquilinear forms and the suitable incorporation of the boundary conditions into an $L^p$-space.

We consider the approach in more detail.
 The boundary part $\overline{\Gamma}$ is Lipschitz regular, and  the interface $\SX\subset \Omega$ 
is a $(d-1)$-set in the sense of Jonsson--Wallin \cite{JW} (cf.\ Assumptions \ref{a-region}
 and \ref{a-d-1}).
For the equations we first treat the case $\varepsilon \equiv 1$, and 
consider the sesquilinear form 
\[
{\gt}[u,v] = \int_\Omega \mu \nabla u \cdot \overline {\nabla v} \, d x,
\]
which is defined on the space $W^{1,2}_\Gamma$ of $W^{1,2}(\Omega)$-functions vanishing on 
$\partial\Omega\setminus \Gamma$.
Note that this reflects the Dirichlet conditions.
For all
$u\in W^{1,2}_\Gamma$ we define the trace $\tr u$ on $\Gamma\cup \Sigma$ in a suitable sense 
(based on \cite{JW}), and show that the map $\gJ u = (u, \tr u)$ is continuous  and has dense range 
from $W^{1,2}_\Gamma$ into $\mathbb L^2 :=L^2(\Omega)\oplus L^2(\Gamma \cup \SX;d\mathcal H_{d-1})$ 
(see Lemma \ref{lLpR202}).
Here $\mathcal H_{d-1}$ is the $(d-1)$-dimensional Hausdorff measure.
These
properties of the trace are a consequence of the regularity of $\Gamma$ and $\SX$.
As the form ${\gt}$ satisfies an ellipticity condition with respect to $\gJ$, the results in \cite{AE2} 
imply that ${\gt}$ induces an operator $A_2$ on $\mathbb L^2$.  
For all $\varphi \in W^{1,2}_\Gamma$ 
and $\mathfrak J \varphi \in \dom (A_2)$ its constitutive relation is 
\begin{equation} \label{e-constitu}
\int_{\Omega \cup\Gamma \cup \SX } (A_2 \mathfrak J \varphi) \,  \mathfrak J\overline \psi \, 
(dx+d\mathcal H_{d-1}) =
\int_\Omega \mu \nabla \varphi \cdot \overline {\nabla \psi} \, d x, \qquad \psi\in W_\Gamma^{1,2}(\Omega). 
\end{equation} 
Let us show that $A_2$ describes the spatial derivatives occurring in 
(\ref{e-parabol}), \eqref{e-robin} and \eqref{e-xi-eq}, respectively, in an adequate manner. 
The argument is heuristic
 in general; moreover we identify within these calculations $\varphi$ with 
$\mathfrak J \varphi$, in order to
make the writing more suggestive. Let 
$\Lambda$ be a surface which is piecewise $C^1$ and which decomposes $\Omega$ into two subdomains $\Omega_1$ 
and $\Omega_2$.
(A prototypical situation is when $\Omega$ is a circular cylinder, $\Gamma$
is its upper plate, and $\Sigma$ is the midplane of $\Omega$.) First put $\Sigma =\Lambda \cap \Omega$ and assume that the outer normal 
$\nu_1$ of $\Omega_1$ across $\Lambda$ equals $\nu_\SX$ on $\partial \Omega_1\cap \SX$.
According to \eqref{e-constitu}, for all $\varphi \in \dom(A_2)$ we have
\[
\int_{\Omega \cup\Gamma \cup \SX } (A_2 \varphi) \, \overline \psi \, (dx+d\mathcal H_{d-1})
= \int_{\Omega_1} \mu \nabla \varphi \cdot \overline{\nabla \psi} \, dx
  + \int_{\Omega_2} \mu \nabla \varphi \cdot \overline{\nabla \psi} \, dx
\]
for all $\psi \in W^{1,2}_\Gamma(\Omega)$.
Since $\psi$ vanishes on $\partial \Omega \setminus \Gamma$, 
one can apply Gauss' theorem to obtain
\begin{equation} \label{e-gauss}
\int_{\Omega_1} \mu \nabla \varphi \cdot \overline{\nabla \psi} \, dx
= \int_{\Omega_1} (-\nabla \cdot \mu \nabla \varphi) \, \overline \psi \, dx
  + \int_{\partial \Omega_1 \cap \Gamma} (\nu \cdot \mu \nabla \varphi) \,  
          \overline \psi \, d\mathcal H_{d-1}
+\int_{\Lambda \cap \Omega} (\nu_1 \cdot \mu \nabla \varphi) \, 
     \overline \psi \, d\mathcal H_{d-1}.
\end{equation}
An equation, analogous to
\eqref{e-gauss}, can also be written for $\Omega_2$.
Then the unit normal $\nu_2$ of $\Omega_2$ across $\Lambda$ equals $-\nu_1$ and one deduces
\begin{eqnarray}
\label{e-compare}
\int_{\Omega \cup\Gamma \cup \SX } (A_2 \varphi) \, \overline \psi \, (dx+d\mathcal H_{d-1})
= \int_{\Omega } (-\nabla \cdot \mu \nabla \varphi) \, \overline \psi \, dx 
   & + & \int_ {\Gamma} (\nu \cdot \mu \nabla \varphi) \,  \overline \psi \, d\mathcal H_{d-1}  \\
& + & \int_{\Lambda \cap \Omega}  [\nu_\SX \cdot  \mu \nabla \varphi] \, \overline \psi \, d\mathcal H_{d-1},
    \nonumber
\end{eqnarray}
where $[\nu_\SX \cdot \mu \nabla \varphi]=\nu_\SX \cdot \big( \mu  \nabla \varphi |_{\partial\Omega_1 \cap \Sigma}
    -  \mu \nabla \varphi |_{\partial\Omega_2 \cap \Sigma}\big)$ 
is the jump in the conormal derivative.
Thus, varying $\psi$ suitably and comparing both sides of \eqref{e-compare},
 one recognizes that $A_2$ has in fact three `components', namely
\begin{enumerate}
\item
the divergence of the vector field 
$\mu \nabla \varphi$ on $\Omega\setminus \SX$, taking $L^2(\Omega) $-functions as values;
\item
the conormal derivative on $\Gamma$, taking $L^2(\Gamma;d\mathcal H_{d-1})$-functions as values;
\item
the jump in the conormal derivative on $\SX$, taking
$L^2(\SX;d\mathcal H_{d-1})$-functions as values.
\end{enumerate}
If one takes $\Sigma$ as a proper subset of $\Lambda \cap \Omega$ (which admits the $(d-1)$-property),
 then \eqref{e-compare} leads to the equation 
\begin{equation*}
\int_{ \SX} (A_2 \varphi) \,  \overline \psi \, d\mathcal H_{d-1}
= \int_{\Lambda \cap \Omega}  [\nu_\SX \cdot  \mu \nabla \varphi] \, \overline \psi \, d\mathcal H_{d-1},
\end{equation*}
which enforces $[\nu_\SX \cdot \mu \nabla \varphi]$ to vanish on 
$(\Lambda \cap \Omega) \setminus \overline{\SX}$. 
Hence the dynamic equations on $\Gamma$ and 
$\SX$ are modelled by the part $L^2(\Gamma \cup \SX; d\mathcal H_{d-1})$ of the base space $\mathbb L^2$.
The subsequent analysis will show that, in either the elliptic or in the parabolic setting,
these three components may be prescribed, and the equation indeed has a solution in the functional
analytic setting which we will establish. 
Moreover, the solution depends continuously on the data.

 The operator $-A_2$ generates a holomorphic, submarkovian $C_0$-semigroup of contractions on $\mathbb L^2$, and may thus be extended to a 
semigroup of contractions on $\mathbb L^p$ for all $p\in [1,\infty]$. 
Denoting the corresponding
generators by $-A_p$, it turns out that for all $p\in (1,\infty)$ the operator $-\varepsilon^{-1} A_p$ 
generates  a holomorphic $C_0$-semigroup of contractions on a suitably renormed $\LL^p$-space.
This has two important consequences.
First, applying an abstract result that is presented e.g.\ in 
\cite[Proposition 2.2]{LeMX}, we obtain a bounded holomorphic functional calculus 
for $\varepsilon^{-1} A_p$
 with angle strictly smaller than $\frac{\pi}{2}$, and in particular the boundedness of
the purely imaginary powers (see Theorem \ref{t-imagin}).
Moreover, the pioneering theorem of Lamberton \cite{Lamb} gives us maximal parabolic regularity 
for $\varepsilon^{-1}A_p$ in Theorem \ref{t-qreg}, which we consider as the main result of this work.
The introduction of temporal weights as in \cite{PrSi} further allows to reduce the regularity of
the initial data almost up to the base space $\mathbb L^p$.
This yields the solution of (\ref{e-parabol})--(\ref{e-initial}) in an
 adequate manner, see Theorem~\ref{t-solution}.

Based on these linear results we treat a nondegenerate quasilinear variant of 
(\ref{e-parabol})--(\ref{e-initial}), even if the right hand side explicitly and 
discontinuously depends on time (Theorem \ref{t-semilinear}).
Here a difficulty is that
 the domain of the realization of the operator $-\nabla \cdot \mu \nabla$ on 
$\mathbb L^p$ is not independent of the coefficients $\mu$.
We therefore consider a problem which is obtained
when applying the Kirchhoff transform to the original one, and which involves only one fixed 
operator (see Definition \ref{quasi-solution}).
 Maximal parabolic regularity then allows to apply a result of Pr\"uss \cite{pru2} 
(see also \cite{CL}) to the transformed problem, giving local existence and uniqueness of solutions 
in a suitable sense. 
Throughout it is essential that 
$\dom (A_p^\theta ) \subset \mathbb L^\infty$ for 
large $p$ and $\theta$ sufficiently close to $1$, which is a consequence of 
ultracontractivity estimates for the semigroup (see Lemma \ref{lLpR206} and 
Proposition \ref{pLpR301}).
The quasilinear problems may be of relevance for
the applications: the heat source on the hypersurface can depend on the solution itself,
 and, additionally, explicitly on time. 

Let us briefly compare the approach in this paper 
with those in \cite{Grie2},  \cite{HaR} and \cite{HaR3} for static Robin boundary conditions.
There the Banach space under consideration is a negatively indexed Sobolev space of type 
$H^{-\theta,q}$ or a Sobolev--Morrey space. In contrast to that settings, in $\mathbb L^p$ one
 may form  the dual pairing of the above parabolic equation with the indicator function $\chi_\Lambda$
of suitable subsets $\Lambda \subset \Omega$. Then one may, additionally, apply Gauss' theorem
to 
$\langle -\nabla \cdot \mu \nabla u,\chi_\Lambda \rangle 
=\int _{\Lambda} -\nabla \cdot \mu \nabla u\,
d x +\int_{\Lambda \cap \Sigma}-\nabla \cdot \mu \nabla u\, d\mathcal H_{d-1}$.
This allows to recover the underlying  physical balance law for the parabolic equation,
which is the starting point for the numerical treatment of such problems. 
For more details we refer to Remark \ref{rem-heuristics-2}.

\medskip

This paper is organized as follows.
In Section \ref{Sec-2} we introduce the spaces 
$\mathbb L^p$, define an appropriate realization of $- \nabla\cdot \mu \nabla$ and show
 that it admits a bounded holomorphic functional calculus.
 In Section \ref{s-parabolic}
 we show that in this setting (\ref{e-parabol})--(\ref{e-initial}) enjoys maximal parabolic
 regularity, and in Section \ref{s-semilinear} we treat the quasilinear case. 
We finish with some concluding remarks in Section \ref{s-conclud}.

\section{Elliptic operators on $\mathbb L^p$}\label{Sec-2}

\subsection{Notation}
\label{s-nota}
Throughout this paper $\mathcal L(X;Y)$ denotes the space of bounded
linear operators from $X$ to $Y$, where $X$ and $Y$ are Banach spaces.
If $X = Y$, then we abbreviate $\mathcal L(X)$.
Note that if $X$ and $Y$ are two Banach
 spaces spaces such that $X \subset Y$ as vector spaces, and both $X$ and $Y$ are
 continuously embedded in a Hausdorff locally convex space, then the inclusion map 
from $X$ into $Y$ is continuous by the closed graph theorem.

In the sequel let $\Omega$ be a bounded domain in $\R^d$ with $d >1$ and 
$\Gamma$ an open part of its boundary $\partial\Omega$, which may be empty.
If $p \in [1,\infty)$, then
$L^p(\Omega)$ is the space of complex-valued, Lebesgue measurable, $p$-integrable
functions on $\Omega$, and for all $\theta \in [0,1]$
we denote by $W^{\theta,p}(\Omega)$  the usual Sobolev--Slobodetskii spaces, 
see \cite{Gris} or \cite{Maz}.
Moreover, $L^\infty(\Omega)$ is the space of Lebesgue measurable,
essentially bounded functions on $\Omega$.
The $(d-1)$-dimensional Hausdorff measure on 
$\R^d$ is denoted by $\mathcal H_{d-1}$.
We denote by $B( x, r)$ the ball in $\R^d$ centred at $ x$ with radius $r$.

\subsection{The function spaces} \label{s-function-spaces} In this subsection we consider the function spaces on which \eqref{e-parabol}--\eqref{e-initial} will be posed.

\begin{definition} \label{d-w01p}
For all $q \in [1,\infty]$ we define $W_\Gamma^{1,q}$ as the closure in 
$W^{1,q}(\Omega)$ of the set
\[
    C_\Gamma^{\infty}(\Omega) 
    \stackrel{\mathrm{def}}{=} 
    \Big\{ u|_{\Omega} : 
    u \in C_c^{\infty}({\R}^d), \, \supp(u) \cap (\partial \Omega \setminus \Gamma)
    = \emptyset 
    \Big\}.
\]
\end{definition}
 
Throughout this paper we make the following assumption on $\Gamma$.

\begin{assu} \label{a-region}
For all $ x \in \overline \Gamma$ there is an open neighbourhood
$\mathcal V_ x$ of $x$ and a bi-Lipschitz mapping 
$F_ x$ from $\mathcal V_ x$ onto the open unit cube $E$ in $\mathbb R^d$, such that 
$F_ x(x) = 0$ and 
$F_ x(\Omega \cap \mathcal V_ x)$ is equal to
the lower open half cube $E_- = (-1,1)^{d-1} \times (-1,0)$ of $E$.
\end{assu}

The reader should notice that the domain $\Omega$ does not need to be Lipschitzian. 
Moreover, nothing is supposed on the boundary of $\Gamma$ within $\partial \Omega$. 

An important technical tool is an extension operator for the $W^{1,q}_\Gamma$-spaces.

\begin{proposition} \label{p-extension}
There is an extension operator $\gE \colon L^1(\Omega) \to L^1(\R^d)$
such that the restriction $\gE|_{W^{1,q}_\Gamma}$ maps $W^{1,q}_\Gamma$ continuously 
into $W^{1,q}(\R^d)$ for all $q \in [1,\infty]$, 
the restriction $\gE|_{L^q(\Omega)}$ maps $L^q(\Omega)$ continuously into 
$L^q(\R^d)$ for all $q \in [1,\infty]$ and 
$\supp \gE u \subset B(0,2R)$ for all $u \in L^1(\Omega)$, where
$R = \sup \{ |x| : x \in \Omega \} $.
\end{proposition}
\begin{proof}
The proof is given in \cite[Lemma~3.4]{ERe1} for the case $q=2$, but carries over to 
all $q \in [1,\infty]$.
Moreover, the second assertion is also easily checked.
The last statement follows by multiplication with a suitable 
$C_c^\infty(\R^d)$-function.
\end{proof}

It turns out that a classical condition from geometric measure theory is 
tailor made in order to define a geometric assumption on a
$(d-1)$-dimensional shape $\SX$ in $\Omega$.

\begin{assu} \label{a-d-1}
Let $\SX\subset \Omega$ be a $(d-1)$-set in the sense of Jonsson--Wallin 
\cite[Subsection~VII.1.1]{JW}.
Precisely: the set $\SX$ is Borel measurable and there exist $c_1,c_2 > 0$ such that
\begin{equation} \label{e-measure00}
 c_1 r^{d-1} \le \mathcal H_{d-1} \bigl (B( x, r) \cap \SX \bigr ) 
\le c_2 r^{d-1} 
\end{equation}
for all $x \in \SX$ and $r \in (0,1)$.
\end{assu}

\begin{rem} \label{r-hypersurf}
We emphasize that  $\SX$ does not have to be closed. Nevertheless 
has $\Sigma$ \emph{finite} $(d-1)$-dimensional Hausdorff measure, according to 
\eqref{e-measure00}. The prototype of $\SX$ is the finite union $\bigcup_j\SX_j$ 
of Lipschitzian hypersurfaces. In that case the restriction of the Hausdorff measure
$\mathcal H_{d-1}$ to $\Gamma$ or to $\SX_j$ can be constructed explicitly in terms
of the local bi-Lipschitz charts (compare \cite[Section~3.3.4~C]{EvG}). In particular,
if $\SX$ is a finite union of Lipschitz graphs, then \eqref{e-measure00} is easily 
verified using this representation of $\mathcal H_{d-1}$.
Moreover, Assumption~\ref{a-d-1} implies for general $\Sigma$
 that  $\SX$ is of ($d$-dimensional) Lebesgue measure $0$.
\end{rem}

Throughout this paper we always presume Assumptions~\ref{a-region} and \ref{a-d-1}.
\begin{definition} We denote by $\rho$ the restriction of the 
Hausdorff measure $\mathcal H_{d-1}$ to $\Gamma \cup \SX$.
\end{definition}

If $u \in L^1_\loc(\R^d)$ and $F \subset \R^d$ is a set, then define 
the function $\tr_F u$ as in \cite[Page~15]{JW} by 
\[
(\tr_F u)(x)
= \lim_{r \to 0} \frac{1}{|B(x,r)|} \, \int_{B(x,r)} u(y)\, d y,
\]
for all $x \in F$ for which the limit exists.
The domain
$\dom(\tr_F u)$ of $\tr_F u$ is the set of all $x \in F$ for which this limit exists.

\begin{lemma} \label{l-measure}
Let $q,r \in [1,\infty)$ 
and $\theta \in [0,1]$.
Let $\gE$ be the extension operator as in Proposition~{\rm \ref{p-extension}}.
\begin{statements}
\item \label{l-measure-1}
If $\frac {1}{q} - \frac {1-\theta}{d} \le \frac {1}{r}$, then
$\mathfrak E$ maps $W^{1,q}_\Gamma$ continuously
 into $W^{\theta,r}(\R^d)$.
\item \label{l-measure-1.5}
If $\frac {1}{q} - \frac {1-\theta}{d} < \frac {1}{r}$, then
$\mathfrak E$ maps $W^{1,q}_\Gamma$ compactly into $W^{\theta,r}(\R^d)$.
\item \label{l-measure-2}
If $\theta \in (\frac {1}{q},1]$, then the trace map 
$u \mapsto \tr_{\Gamma \cup \SX} u$ is continuous from $W^{\theta,q}(\R^d)$ 
into $L^q(\Gamma \cup \SX; d\rho)$.
\end{statements}
\end{lemma}
\begin{proof}
`\ref{l-measure-1}' and `\ref{l-measure-1.5}'. 
This follows from Proposition~\ref{p-extension}, the support property of $\mathfrak E$ and the usual Sobolev embedding.

`\ref{l-measure-2}'. 
Since $\Gamma$ and $\SX$ are disjoint, the natural 
map from the space $L^q(\Gamma \cup \SX; d\rho)$
into $L^q(\SX ; d\mathcal H_{d-1})\times  L^q( \Gamma; d\mathcal H_{d-1})$
is a linear, topological isomorphism.
Therefore, it suffices to show that
the trace maps 
$u \mapsto \tr_\Gamma u$ and $u \mapsto \tr_\SX u$
are continuous from $W^{\theta,q}(\R^d)$ 
into  $L^q(\Gamma; d\mathcal H_{d-1})$ and $L^q(\SX; d\mathcal H_{d-1})$.

It follows from 
\cite[Chapter~VIII, Proposition~1]{JW} that property \eqref{e-measure00} inherits
to the closure $\overline \SX$ of $ \SX$.
Then the trace operator $u \mapsto \tr_{\overline \SX} u$
is bounded from $W^{\theta,q}(\R^d)$ into $L^q(\overline { \SX}; d\mathcal H_{d-1})$
by  \cite[Chapter~V, Theorem~1]{JW}. 
But the set difference 
$\overline { \SX}\setminus  \SX$ is of $\mathcal H_{d-1}$
measure $0$ (see again \cite[Chapter~VIII, Proposition~1]{JW}).
Consequently the spaces  
$L^q(\overline { \SX}; d\mathcal H_{d-1})$ and 
$L^q( { \SX}; d\mathcal H_{d-1})$ are identical. 

Next we consider the set $\Gamma$.
Using the notation as in Assumption \ref{a-region}, 
for every  $x \in \overline \Gamma$ the map
$F_x$ provides a bi-Lipschitz parametrization of $\partial \Omega \cap \mathcal V_x$, where 
the parameters run through the upper plate 
$P:= (-1,1)^{d-1} \times \{0\}$ of the half cube $E_-$.
Moreover, the Hausdorff measure
$\mathcal H_{d-1}$ on $\partial \Omega \cap \mathcal V_x$ is the 
surface measure, and the latter is obtained from the Lebesgue measure on $(-1,1)^{d-1}
 \times \{0\}$ via the bi-Lipschitzian parametrization, see \cite[Section~3.3.4~C]{EvG}.
Define $\mathcal W_x = F_x\big ( (-\frac {1}{2},\frac {1}{2})^{d-1} \times \{0\}\big)$.
Then  $\mathcal W_x \subset \partial \Omega$.
There exist $n \in \N$ and $x_1,\ldots,x_n \in \overline \Gamma$ such that 
$\mathcal W_{x_1},\ldots,\mathcal W_{x_n}$ is a finite cover of $\overline \Gamma$.
Obviously, 
$\overline {\mathcal W_{x_1}},\ldots,\overline {\mathcal W_{x_n}} $ 
is also a finite cover of $\overline \Gamma$.
Moreover, it is not hard to see that $\bigcup_{j=1}^n \overline {\mathcal W_{x_j}}$ 
is a $(d-1)$-set in the sense of Jonsson--Wallin (compare \cite[Lemma~3.2]{HaR2}).
Hence by \cite[Chapter~V, Theorem~1]{JW} there exists
a continuous trace operator from $W^{\theta,q}(\R^d)$ into 
$L^q(\cup_{j=1}^n \overline {\mathcal W_{x_j}}; d\mathcal H_{d-1})$.
Combining
this operator with the restriction operator to $\Gamma$, one obtains the desired 
trace operator into $L^q(\Gamma ; d\mathcal H_{d-1})$.
\end{proof}

For all $u \in L^1_\loc(\Omega)$ define the function $\tr u$ as in \cite[Section~VIII.1.1]{JW} by 
\[
\dom(\tr u) 
= \Big\{ x \in \Gamma \cup \SX : 
    \lim_{r \to 0} \frac{1}{|B(x,r) \cap \Omega|} \, \int_{B(x,r) \cap \Omega} u(y)\, dy \;\;
   \mbox{ exists}
  \Big\}
\]
and 
\[
(\tr u)(x) 
= \lim_{r \to 0} \frac{1}{|B(x,r) \cap \Omega|} \, \int_{B(x,r) \cap \Omega} u(y)\,dy
\]
for all $x \in \dom(\tr u)$.

The above defined trace enjoys the following mapping properties.

\begin{proposition} \label{pLpR201}
Let $q,r \in (1,\infty)$ and suppose that 
$\frac{d-q}{q} < \frac{d-1}{r}$.
Then $\tr u \in L^r(\Gamma \cup \SX; d\rho)$ for all $u \in W^{1,q}_\Gamma$,
and the map 
$u \mapsto \tr u$ is compact from $W^{1,q}_\Gamma$ into $L^r(\Gamma \cup \SX; d\rho)$.
\end{proposition}
\begin{proof}
Let $\gE$ be the extension operator as in Proposition~{\rm \ref{p-extension}}.
Then it follows from Lemma~\ref{l-measure} that $u \mapsto \tr_{\Gamma \cup \SX} \gE u$ 
maps $W^{1,q}_\Gamma$ compactly into $L^r(\Gamma \cup \SX; d\rho)$.
But if $u \in W^{1,q}_\Gamma$, then we claim that
\begin{equation}\label{trace-identity}
(\tr u)(x) = (\tr_{\Gamma \cup \SX} \gE u)(x)
\end{equation}
for $\mathcal H_{d-1}$-a.e.\ $x \in \Gamma \cup \SX$.
Obviously, this identity holds for 
$\mathcal H_{d-1}$-a.e.\ $x\in \SX$ since $\SX \subset \Omega$.
For $\mathcal H_{d-1}$-a.e.\ $x\in \Gamma$ we 
can argue as in the proof of \cite[Chapter~VIII, Proposition~2]{JW}, where the case 
$\Gamma = \partial\Omega$ is considered. Indeed, the arguments given there are purely local.
Since $\gE u \in W^{1,q}(\R^d)$ it follows that for $\mathcal H_{d-1}$-a.e.\ $x\in \Gamma$
there exists a Borel set $E \subset \R^d$ such that 
$\mathcal H_{d-1}(E\cap B(x,r)) = o(r^{d-1})$ and 
$(\gE u)(x) = \displaystyle \lim_{y\to x, \; y\notin E} (\gE u)(y)$.
 Using these properties of $E$, the same arguments as in the last part of the proof given 
in \cite{JW} establish \eqref{trace-identity}.
\end{proof}

The space on which (\ref{e-parabol})--(\ref{e-initial}) will be realized is given as follows.

\begin{definition} \label{d-lp}
For all $p \in [1,\infty]$, denote by 
$\mathbb L^p$ the Lebesgue space $L^p(\Omega \cup \Gamma; d x+d\rho)$.
We denote the space of all real valued functions in $\mathbb L^p$ by $\mathbb L^p_\R$.
\end{definition}

Observe that there is a natural topological isomorphism
between $\mathbb L^p$ and the direct sum $L^p(\Omega) \oplus L^p(\Gamma\cup \SX;d\rho)$
and we will identify $\mathbb L^p$ with $L^p(\Omega) \oplus L^p(\Gamma\cup \SX;d\rho)$
through this natural map.

By Proposition~\ref{pLpR201} we can define the map $\gJ \colon W^{1,2}_\Gamma \to \LL^2$
by 
\[
\gJ u = (u, \tr u) \in L^2(\Omega) \oplus L^2(\Gamma\cup \SX;d\rho) \cong \LL^2
 . \]
Note that one can always choose some $p > 2$ in Statement~\ref{lLpR202-2} of the next lemma.

\begin{lemma} \label{lLpR202}
\mbox{}
\begin{statements}
\item \label{lLpR202-1}
The map $\gJ$ is continuous and has dense range.
\item \label{lLpR202-2}
If $p \in [1,\infty)$ and $(d-2) p < 2 (d-1)$, then 
$\gJ W^{1,2}_\Gamma \subset \LL^p$.
\item \label{lLpR202-3}
The map $\gJ$ is compact.
\end{statements}
\end{lemma}
\begin{proof}
`\ref{lLpR202-1}'.
The continuity follows from Proposition~\ref{pLpR201}.
Let $f = (f_\Omega,f_\partial) \in L^2(\Omega) \oplus L^2(\Gamma\cup \SX;d\rho)$
and suppose that 
$(\gJ u,f)_{L^2(\Omega) \oplus L^2(\Gamma\cup \SX;d\rho)} = 0$ for all $u \in W^{1,2}_\Gamma$.
We show that $f=0$.
For all $u \in C_c^\infty(\Omega \setminus \overline \SX)$ one has 
$0 = (\gJ u,f) = \int_\Omega u \, \overline{f_\Omega}\, dx$.
Since $C_c^\infty(\Omega \setminus \overline \SX)$ is dense in 
$L^2(\Omega \setminus \overline \SX) = L^2(\Omega)$ one deduces that 
$f_\Omega = 0$.
Therefore 
$0 = \int_{\Gamma\cup \SX} \tr u \, \overline{f_\partial} \, d\rho$
for all $u \in W^{1,2}_\Gamma$ and in particular for all 
$u \in C^\infty_\Gamma(\Omega)$.
But $ \{ u|_{\Gamma\cup \SX} : u \in C^\infty_\Gamma(\Omega) \} $ is 
dense in $L^2(\Gamma\cup \SX; d\rho)$.
So $f_\partial = 0$.

`\ref{lLpR202-2}'.
If $\gE$ is the extension operator as in Proposition~{\rm \ref{p-extension}}
then it follows from Lemma~\ref{l-measure}\ref{l-measure-1} that
$\gE$ maps $W^{1,2}_\Gamma$ continuously into $L^p(\R^d)$
for all $p \in [1,\infty)$ with $(d-2) p \leq 2d$.
So $W^{1,2}_\Gamma \subset L^p(\Omega)$.
Now the statement follows from Proposition~\ref{pLpR201}.

`\ref{lLpR202-3}'.
It follows immediately from Lemma~\ref{l-measure}\ref{l-measure-1.5}
that the restriction $\gE|_\Omega$ maps $W^{1,2}_\Gamma$
compactly into $L^2(\Omega)$.
So the embedding of $W^{1,2}_\Gamma$ into $L^2(\Omega)$ is compact.
Also the map $\tr$ is compact from $W^{1,2}_\Gamma$ into 
$L^2(\Gamma\cup \SX;d\rho)$ by Proposition~\ref{pLpR201}.
Therefore the map $\gJ$ is compact.
\end{proof}

We end this subsection with a truncation lemma.

\begin{lemma} \label{lLpR204} 
Let $u \in W^{1,2}_\Gamma$ be real-valued.
Then $u \wedge \one_\Omega \in W^{1,2}_\Gamma$ and
$\gJ(u \wedge \one_\Omega) = (\gJ u) \wedge \one_{\Omega \cup \Gamma}$.
\end{lemma}
\begin{proof}
The first statement is shown in the proof of \cite[Theorem~3.1]{ERe1}.
The second statement is obvious for real-valued $u\in C_\Gamma^\infty(\Omega)$. 
Since the maps $u\mapsto \gJ(u \wedge \one_\Omega)$ and 
$u\mapsto (\gJ u) \wedge \one_{\Omega \cup \Gamma}$ are continuous on the real version of $W^{1,2}_\Gamma$,
the identity carries over to the general case by density.
\end{proof}

\subsection{The operator on $\LL^p$} \label{SLpS2.2}

In this subsection we introduce a differential operator on $\LL^p$ that corresponds to the spatial derivatives in \eqref{e-parabol}, \eqref{e-robin} and \eqref{e-xi-eq}.

Throughout the remaining of this paper we adopt the next assumption.

\begin{assu} \label{a-coeff}
  Let $\mu=\bigl \{\mu_{k,l}\bigr \}_{k,l} \colon \Omega \to
\mathcal L(\R^d;\R^d)$ be a measurable map from $\Omega$ into the set of real
 $d\times d$ matrices.
We assume that there are $\low{\mu}, \upp{\mu} > 0$ such that
\[
 \norm{{\mu}( x)}_{\mathcal L(\R^d;\R^d)}
      \le
      \upp{\mu}\quad \text{and}    
\quad \sum_{k,l=1}^d \mu_{k,l}( x)\, \xi_k\, \xi_l\ge 
      \low{\mu} \sum_{k=1}^d \xi_k^2
\]
   for all $ x \in \Omega $ and $\xi=(\xi_1,\ldots,\xi_d)\in\R^d$.
\end{assu}

We emphasize that $\mu$ does not have to be symmetric.

\begin{definition}  \label{d-form}                                
 Define the sesquilinear form ${\gt} \colon W^{1,2}_\Gamma \times W^{1,2}_\Gamma \to \C$
by 
\[
      {\gt}[u,v]  
=      \int_\Omega 
      \mu \nabla u \cdot \overline {\nabla v} \, d x
 .
\]
\end{definition}

We emphasize that the domain of the form $\gt$ is the space $W^{1,2}_\Gamma$,
 which appropriately incorporates the Dirichlet condition on 
$\partial \Omega \setminus \Gamma$, compare \cite[Section~1.2]{Cia} or \cite[Section~II.2]{GGZ}.
The form $\gt$ is continuous and 
\begin{equation}\label{J-ellipticity}
\RRe \gt[u,u] + \|\gJ u\|_{\LL^2}^2 
\geq (\low{\mu} \wedge 1) \|u\|_{W^{1,2}_\Gamma}^2
\end{equation}
for all $u \in W^{1,2}_\Gamma$.
Therefore by Lemma  \ref{lLpR202}\ref{lLpR202-1} and \cite[Theorem~2.1]{AE2} there exists a 
unique operator $A_2$ in $\LL^2$ such that for all $\varphi,\psi \in \LL^2$ one has 
$\varphi \in \dom(A_2)$ and $A_2 \varphi = \psi$ if and only if there exists a 
$u \in W^{1,2}_\Gamma$ such that $\gJ u = \varphi$ and 
\begin{equation} \label{e-opdef}
\gt[u,v]
= (\psi, \gJ v)_{\LL^2}
\end{equation}
for all $v \in W^{1,2}_\Gamma$.
Although the form domain of $\gt$ is $W^{1,2}_\Gamma$, the operator $A_2$ 
is an operator in $\LL^2$. 
We refer to the introduction for a discussion of 
the relation of $A_2$ to the original problem  (\ref{e-parabol})--(\ref{e-initial}).

\begin{rem} \label{r-identopJ}
The construction of $A_2$ generalizes the derivation of an operator from a suitable form
$\mathfrak s$ to the case when the form domain $D_\mathfrak s$ is a priori \emph{not} contained in 
the corresponding Hilbert space $\mathfrak H$ (compare \cite[Section~VI.2]{Kat1} for the classical case).
The substitute for the inclusion $D_\mathfrak s \subset \mathfrak H$ is the definition of 
an appropriate embedding operator $\gJ \colon D_\mathfrak s \to \mathfrak H$.
Fortunately, all tools for form methods are still available.
\end{rem}
\begin{proposition} \label{pLpR203}
The operator $A_2$ is $m$-sectorial with vertex $0$ and semi-angle 
$\arctan \frac{\upp{\mu}}{\low{\mu}}$.
Moreover, $A_2$ has compact resolvent.
\end{proposition}
\begin{proof}
It follows from \cite[Theorem~2.1]{AE2} that $A_2$ is $m$-sectorial. Let $\varphi \in \dom(A_2)$
 and $u \in W^{1,2}_\Gamma$ with $\gJ u = \varphi$. Then $\RRe ( A_2 \varphi, \varphi)_{\LL^2} 
= \RRe  \gt[u,u]  \geq 0$. Hence the vertex is $0$. Further, one has $\RRe \gt[u,u] \geq
 \low{\mu} \int_\Omega |\nabla \RRe u|^2 +|\nabla \IIm u|^2\, dx $ and 
\[
|\!\IIm \gt[u,u]| \leq 2 \upp{\mu} 
\int_\Omega |\nabla \RRe u| |\nabla \IIm u| \, dx \leq \upp{\mu}
\int_\Omega |\nabla \RRe u|^2 +|\nabla \IIm u|^2\, dx.
\]
Thus $|\arg ( A_2 \varphi, \varphi)_{\LL^2} | \leq \arctan \frac{\upp{\mu}}{\low{\mu}}$
if $\varphi \neq 0$.

Since the map $\gJ$ is compact by Lemma~\ref{lLpR202}\ref{lLpR202-3},
the generator has compact resolvent by \cite[Lemma~2.7]{AE2}.\end{proof}

We continue with the analysis of the operator $A_2$.
By Proposition~\ref{pLpR203} the operator $A_2$ is $m$-sectorial with vertex $0$ and semi-angle
$\arctan \frac{\upp{\mu}}{\low{\mu}}$.
Hence by \cite[Theorem IX.1.24]{Kat1} the operator $-A_2$ generates a 
holomorphic semigroup, denoted by $S$, which is holomorphic and contractive on the 
sector with semi-angle $\arctan \frac{\upp{\mu}}{\low{\mu}}$.

\begin{proposition} \label{pLpR205}
The semigroup $S$ leaves $\LL^2_\R$ invariant, it is submarkovian and positive.
\end{proposition}
\begin{proof} 
Clearly the set $\LL^2_\R$ is closed and convex in $\LL^2$.
Moreover, $\varphi \mapsto \RRe \varphi$ is the projection from 
$\LL^2$ onto $\LL^2_\R$ and $\RRe \gt (u, u - \RRe u) = 0$ for all 
$u \in W^{1,2}_\Gamma$.
Since the form $\gt$ is accretive, the set $\LL^2_\R$ is invariant under the semigroup by 
\cite[Proposition~2.9(ii)]{AE2}

Next, let $C = \{ u \in \LL^2 : u \mbox{ is real valued and } u \leq \one \} $.
Then $C$ is closed and convex.
Let $P \colon \LL^2 \to C$ denote the orthogonal projection.
Then $P u = (\RRe u) \wedge \one_{\Omega \cup \Gamma}$.
Let $u \in W^{1,2}_\Gamma$.
By Lemma~\ref{lLpR204} one has  
$(\RRe u) \wedge \one_\Omega \in W^{1,2}_\Gamma$ and $P\gJ u = \gJ((\RRe u) \wedge \one_\Omega)$.
Moreover, an easy calculation gives
\[
\RRe \gt[(\RRe u) \wedge \one_\Omega, u - (\RRe u) \wedge \one_\Omega] = 0  
 .  \]
Observing that the form $\gt$ is accretive, it follows from \cite[Proposition~2.9(ii)]{AE2} 
that $C$ is invariant under the semigroup $S$.
Now let $\varphi\in \LL^2 \cap \LL^\infty$ and $t > 0$.
There exists an $\alpha \in \R$ such that 
$\|S_t \varphi\|_{\LL^\infty} = \|\RRe (e^{i \alpha} S_t \varphi)\|_{\LL^\infty}$. 
But 
$\RRe (e^{i \alpha} S_t \varphi)
= S_t \RRe (e^{i \alpha} \varphi)$.
Therefore 
\[
\|S_t \varphi\|_{\LL^\infty}
= \|S_t \RRe (e^{i \alpha} \varphi)\|_{\LL^\infty}
\leq \|\RRe (e^{i \alpha} \varphi)\|_{\LL^\infty}
\leq \|\varphi\|_{\LL^\infty}
\]
and $S$ is submarkovian.

Finally, if $\varphi \in \LL^2_\R$ and $\varphi \leq 0$, then $n \varphi \in C$
for all $n \in \N$.
So $S_t(n \varphi) \leq \one$ for all $t > 0$ and $n \in \N$.
Therefore $S_t \varphi \leq 0$ and $S$ is positive.
\end{proof}

\begin{corollary} \label{cLpR209}
For all $p \in [1,\infty]$ the semigroup $S$ extends consistently 
to a contraction semigroup $S^{(p)}$ on $\LL^p$.
The semigroup $S^{(p)}$ is strongly continuous for all $p \in [1,\infty)$.
\end{corollary}
\begin{proof}  
Observe that if the coefficient matrix $\mu$ satisfies the conditions 
of Assumption~\ref{a-coeff}, then its transpose satisfies these as well. 
Thus the dual semigroup $S^*$ shares the same properties as $S$. 
Now the assertion follows from Proposition \ref{pLpR205} and standard interpolation and duality arguments, 
see e.g.\ \cite[page 56]{Ouh5}. 
\end{proof}

We denote the generator of $S^{(p)}$ by $-A_p$.
Then $-A_p$ is dissipative by the Lumer--Phillips theorem.
If no confusion is possible we write $S = S^{(p)}$.
\begin{rem} \label{r-invers}
It is possible to prove the dissipativity of $-A_p$ also by showing that the form
$-\mathfrak t$ is $p$-dissipative, cf. \cite{CiaM}. 
\end{rem}

\begin{lemma} \label{lLpR206}
\mbox{}
\begin{statements}
\item \label{lLpR206-1}
The semigroup $S$ is ultracontractive.
Moreover, for all $\beta > d-1$  and $\omega >0$ there exists a $c > 0$ such that 
\[
\|S_t \varphi\|_{\LL^q} 
\leq c \, t^{-\beta ( \frac{1}{p} - \frac{1}{q} )} e^{\omega t}\|\varphi\|_{\LL^p}
\]
for all $t > 0$, $\varphi \in \LL^p$ and $p,q \in [1,\infty]$ with $p \leq q$.
\item \label{lLpR206-2}
If $1 \leq p < q \leq \infty$ and $j \in \N$ are such that 
$\frac{d-1}{j} \, ( \frac{1}{p} - \frac{1}{q} ) < 1$, then 
the operator $(A_p + 1)^{-j}$ maps $\LL^p$ continuously into $\LL^q$.
\item \label{lLpR206-3}
The operator $A_p$ has compact resolvent for all $p \in (1,\infty)$.
\item \label{lLpR206-4}
If the matrix of coefficients $\mu$ is symmetric, then the operator $A_2$ is 
self-adjoint and positive.
\end{statements}
\end{lemma}
\begin{proof}
`\ref{lLpR206-1}'.
Let $r \in (2,\infty)$ be such that $(d-2) r < 2(d-1)$.
Then it follows from Lemma~\ref{lLpR202}\ref{lLpR202-2} that 
$ \gJ W^{1,2}_\Gamma \subset \LL^r$, and the inclusion is continuous
by the closed graph theorem.
Let $\varphi \in \mathbb L^2$ and $t>0$.
Since $S_t \varphi\in \dom (A_2)$, there is a $u \in W^{1,2}_\Gamma$ such that $S_t \varphi = \gJ u$.
For given $\omega >0$ one has
\begin{eqnarray*}
\|S_t \varphi\|_{\mathbb L^r}^2  
= \|\gJ u\|_{\mathbb L^r}^2  \leq C\,\|u\|_{W^{1,2}_\Gamma}^2 
& \leq & C (\low{\mu} \wedge 1)^{-1} \big( \RRe \gt[u,u] + \|\gJ u\|_{\LL^2}^2 \big)\\
& = & C (\low{\mu} \wedge 1)^{-1}  \big( \RRe ( A_2 S_t \varphi,  S_t \varphi)_{\mathbb L^2} + \|S_t \varphi\|_{\LL^2}^2 \big)\\
& \leq & C' \, t^{-1} e^{2\omega t}\|\varphi\|_{\LL^2}^2,
\end{eqnarray*}
for suitable $C,C' > 0$, 
using \eqref{J-ellipticity}, the definition of $A_2$, the Cauchy--Schwarz inequality 
and the holomorphy and contractivity of $S_t$.
Therefore the semigroup $S$ is ultracontractive, and by \cite[Lemma~6.1]{Ouh5} there exists 
a $c > 0$ such that 
\[
\|e^{-\omega t}S_t \varphi\|_{\LL^\infty} \leq c \, t^{- \frac{r}{2(r-2)} } \, \|\varphi\|_{\LL^2}
\]
for all $t > 0$ and $\varphi \in \LL^2$.
Now duality and interpolation give Statement~\ref{lLpR206-1}.

Statement~\ref{lLpR206-2} follows from \ref{lLpR206-1} and the well-known formula 
\[
(A_p+1)^{-j} = \frac{1}{(j-1)!}\int_0^\infty t^{j-1}e^{-t} S_t \, dt.
\]
Statement~\ref{lLpR206-3} is a consequence of Proposition~\ref{pLpR203}
and interpolation.
The last statement of the lemma is easy to verify.
\end{proof}

\subsection{Multipliers acting on Lebesgue spaces}
\label{s-mult}
In order to solve \eqref{e-parabol}--\eqref{e-initial}, we divide \eqref{e-parabol}
 (at first formally) by $\varepsilon$.
Obviously, one is then confronted with the necessity
 to investigate the functional analytic properties of operators of the type $\varsigma A_p$, 
where $\varsigma$ is a bounded strictly positive measurable function.
Concerning the generator property of an analytic semigroup in a space $L^p(\Omega)$
 this was carried out in \cite{GKR} and concerning maximal parabolic regularity
on $L^p(\Omega)$ in \cite{HiebR}.
In the latter case the decisive instrument was 
the insight from \cite{DO} that a suitable multiplicative perturbation does
 not destroy upper Gaussian estimates, which in turn imply maximal parabolic 
regularity on $L^p(\Omega)$.
Unfortunately, we cannot apply this here, since our
 Lebesgue space does not only live `on the volume'.
But a surprisingly simple 
trick allows us to overcome the problem in the present context.

The next proposition is of independent interest.

\begin{proposition} \label{pLpR211}
Let $(X,\cb,\lambda)$ be a measure space and let $\tau \colon X \to (0,\infty)$ be 
a measurable function such that both $\tau$ and $\tau^{-1}$ are bounded.
Let $p \in [1,\infty)$ and let $T$ be an operator in $L^p(X,d\lambda)$.
\begin{statements}
\item \label{pLpR211-1}
If $T$ is dissipative on $L^p(X,d\lambda)$, then $\tau T$ is dissipative on 
$L^p(X,\tau^{-1} d\lambda)$.
\item \label{pLpR211-2}
If $T$ generates a strongly continuous contraction semigroup on $L^p(X,d\lambda)$, 
then $\tau T$ generates a strongly continuous contraction semigroup on 
$L^p(X,\tau^{-1} d\lambda)$.
\item \label{pLpR211-2.5}
If $p = 2$, $\theta \in (0,\frac{\pi}{2}]$ and 
$T$ generates a holomorphic semigroup in $L^2(X,d\lambda)$ which is contractive in the 
sector with semi-angle $\theta$, 
then $\tau T$ generates a holomorphic semigroup in $L^2(X,\tau^{-1} d\lambda)$
which is contractive in the sector with semi-angle $\theta$.
\end{statements}
Now suppose that $p = 2$ and $T$ generates a strongly continuous contraction semigroup
$S$ on $L^2(X,d\lambda)$.
Denote the semigroup generated by $\tau T$ on $L^2(X,\tau^{-1} d\lambda)$ by $S^\tau$.
\begin{statements}
\setcounter{teller}{3}
\item \label{pLpR211-3}
If $S$ leaves the real valued functions invariant, then $S^\tau$ also leaves the 
real valued functions invariant.
\item \label{pLpR211-4}
If $S$ is positive, then $S^\tau$ is also positive.
\item \label{pLpR211-5}
Suppose $S$ is submarkovian.
Then $S^\tau$ is also submarkovian.
Hence for all $q \in [2,\infty)$ the semigroups $S$ and $S^\tau$ extend consistently
to a strongly continuous semigroup $S^{(q)}$ and $S^{(\tau,q)}$ on 
$L^q(X,d\lambda)$ and $L^q(X,\tau^{-1} d\lambda)$, respectively.
Let $T_q$ and $T_{\tau,q}$ denote the generators.
Then $T_{\tau,q} = \tau T_q$ for all $q \in [2,\infty)$.
\end{statements}
\end{proposition}
\begin{proof}
`\ref{pLpR211-1}'.
The operator $T$ is dissipative on $L^p(X,d\lambda)$
if and only if 
\[
\RRe \int_{\{f \neq 0\}} (T f) \, |f|^{p-2} \, \overline f \, d\lambda 
\leq 0
\]
for all $f \in D(T)$.
This implies the dissipativity of $\tau T$ on $L^p(X,\tau^{-1} d\lambda)$.

`\ref{pLpR211-2}'.
Since $T$ generates a contraction semigroup on $L^p(X,d\lambda)$, it follows that 
$T$ is dissipative.
Therefore $\tau T$ is dissipative on $L^p(X,\tau^{-1} d\lambda)$
by Statement~\ref{pLpR211-1}.
So by the Lumer--Phillips theorem it remains to show that 
the operator $\tau T - 1$ is surjective on $L^p(X,\tau^{-1} d\lambda)$.

Let $\delta > 0$ be such that $\tau^{-1} - \delta \geq 0$.
Then the multiplication operator $-( \tau^{-1} - \delta)$ is dissipative on $L^p(X,d\lambda)$ 
and has a relative bound equal to zero with respect to $T$.
Therefore $T - ( \tau^{-1} - \delta)$ generates a strongly continuous contraction semigroup on 
$L^p(X,d\lambda)$ by the perturbation result \cite[Theorem~3.7]{Dav1}.
Hence $T - \tau^{-1}$ is surjective on $L^p(X,d\lambda)$ by the Lumer--Phillips theorem.
But this implies that $\tau T - 1$ is surjective on $L^p(X,\tau^{-1} d\lambda)$. 

`\ref{pLpR211-2.5}'.
For all $\alpha \in (-\theta,\theta)$ the above applies to the 
operator $e^{i \alpha} T$.
Therefore $e^{i \alpha} \tau T$ generates a strongly continuous contraction semigroup 
on $L^2(X,\tau^{-1} d\lambda)$.
Hence by \cite[Theorem IX.1.23]{Kat1}
the operator $\tau T$ generates a holomorphic semigroup in $L^2(X,\tau^{-1} d\lambda)$ which is 
contractive on the sector with semi-angle $\theta$.

Now suppose $p = 2$ and $T$ generates a strongly continuous contraction semigroup
$S$ on $L^2(X,d\lambda)$.
Let $C$ be a closed convex subset of $L^2(X,d\lambda)$.
Then $C$ is also closed and convex in $L^2(X,\tau^{-1} d\lambda)$.
Since $T$ is $m$-dissipative it follows from \cite[Theorem~2.2]{Ouh3}
that $C$ is invariant under $S$ if and only if 
$\RRe (T f, f - Pf)_{L^2(X,d\lambda)} \leq 0$ for all $f \in D(T)$, 
where $P$ is the orthogonal projection in $L^2(X,d\lambda)$ onto~$C$.
Similarly, since $\tau T$ is $m$-dissipative on $L^2(X,\tau^{-1} d\lambda)$, the set
$C$ is invariant under $S^\tau$ if and only if 
$\RRe (\tau T f, f - P^\tau f)_{L^2(X, \tau^{-1} d\lambda)} \leq 0$ for all $f \in D(\tau T)$, 
where $P^\tau$ is the orthogonal projection in $L^2(X,\tau^{-1} d\lambda)$ onto~$C$.
But $D(\tau T) = D(T)$. 
Hence if $P = P^\tau$, then $C$ is invariant under $S$ if and only if $C$ 
is invariant under $S^\tau$.

Then for the proof of Statement~\ref{pLpR211-3} choose 
$C = \{ f \in L^2(X,d\lambda) : f \mbox{ is real valued} \} $ and note that the 
projection is $P f = \RRe f = P^\tau f$.
For the proof of Statement~\ref{pLpR211-4} choose 
$C = \{ f \in L^2(X,d\lambda) : f \mbox{ is real valued and } f \geq 0 \} $ and note that the 
projection is $P f = (\RRe f)^+ = P^\tau f$.
For the submarkovian part in the proof of Statement~\ref{pLpR211-5} choose 
$C = \{ f \in L^2(X,d\lambda) : |f| \leq 1 \mbox{ a.e.} \} $ and note that the 
projection is $P f = (|f| \wedge \one) \sgn f = P^\tau f$.

It remains to prove the second part of Statement~\ref{pLpR211-5}.
Let $q \in [2,\infty)$.
Let $u \in D(T_{\tau,q}) \cap D(T_{\tau,2})$.
Write $v = T_{\tau,2} u$.
Then $u \in L^2(X, d\lambda) \cap L^q(X, d\lambda)$ and 
$T_{\tau,q} u = T_{\tau,2} u = v$.
So $v \in L^q(X, d\lambda)$ and $\tau^{-1} v \in L^q(X, d\lambda)$ since $\tau^{-1}$ is bounded.
Moreover, $T_{\tau,2} = \tau T_2$, so $u \in D(T_2)$ and $T_2 u = \tau^{-1} v$. Therefore
\[
 t^{-1} (S_t^{(q)} u - u ) 
= t^{-1} (S_t^{(2)} u - u ) = t^{-1} \int_0^t S_s^{(2)} T_2 u \, ds 
= t^{-1} \int_0^t S_s^{(q)} T_2 u \, ds
\]
for $t >0$. 
As $t\downarrow 0$, the latter term converges to $T_2 u$ in $L^q(X, d\lambda)$ by the strong continuity of $S^{(q)}$. Hence $u \in D(T_q)$ and $T_q u = T_2 u = \tau^{-1} v$. Then $\tau T_q u = v = T_{\tau,q} u$.
We proved that 
\[
D(T_{\tau,q}) \cap D(T_{\tau,2}) 
\subset D(\tau T_q) \cap D(\tau T_2)
\]
and $T_{\tau,q} u = \tau T_q u$ for all $u \in D(T_{\tau,q}) \cap D(T_{\tau,2})$.
Similarly the converse inclusion is valid, so 
\[
D(T_{\tau,q}) \cap D(T_{\tau,2}) 
= D(\tau T_q) \cap D(\tau T_2)
= D(T_q) \cap D(T_2)
 .  \]
We claim that $D(T_{q}) \cap D(T_{2})$ is dense in $D(T_{q})= D(\tau T_q)$. 
Consider the set 
\[
\cd = \{ t^{-1} \int_0^{t} S_s^{(q)} u \, ds : u \in L^q(X, d\lambda)\cap L^2(X, d\lambda),
                                               t \in (0,\infty) \} 
 .  \]
Then $\cd \subset D(T_q)$.
Since $S^{(q)}$ and $S^{(2)}$ are consistent, also $\cd \subset D(T_2)$.
So $\cd \subset D(T_q) \cap D(T_2)$.
Moreover, $\lim_{t \downarrow 0} t^{-1} \int_0^{t} S_s^{(q)} u \, ds = u$ in $L^q(X, d\lambda)$
for all $u \in L^q(X, d\lambda)\cap L^2(X, d\lambda)$ and 
$L^q(X, d\lambda)\cap L^2(X, d\lambda)$ is dense in $L^q(X, d\lambda)$.
Therefore $\cd$ is dense in $L^q(X, d\lambda)$.
Clearly $\cd$ is invariant under $S^{(q)}$.
Hence $\cd$ is a core for $T_q$ by \cite[Proposition II.1.7]{EN}.
This implies that $D(T_{q}) \cap D(T_{2})$ is dense in $D(T_{q})$.
The same arguments show that 
$D(T_{\tau,q}) \cap D(T_{\tau,2})$ is dense in $D(T_{\tau,q})$.
Hence $T_{\tau,q} = \tau T_q$.
\end{proof}

Let $\varsigma \colon \Omega \cup \Gamma \to (0,\infty)$ be a measurable 
function such that $\varsigma,\varsigma^{-1} \in \LL^\infty$.
We write 
\[
\LL^p_\varsigma
:= L^p(\Omega \cup \Gamma; \varsigma^{-1} (d x +d\rho))
 .  \]
Proposition~\ref{pLpR211} allows to transfer the conclusion of Corollary~\ref{cLpR209}
to the operators $\varsigma A_p$.

\begin{theorem} \label{t-mult03-new}
For all $p \in [1,\infty)$ the operator $-\varsigma A_p$ generates a strongly continuous 
positive semigroup $S^{(\varsigma,p)}$ of contractions on the space $\LL_\varsigma^p$.
The semigroups are consistent. 
Moreover, $S^{(\varsigma,2)}$ is holomorphic and contractive in the sector with 
semi-angle $\arctan \frac{\upp{\mu}}{\low{\mu}}$.
\end{theorem}
\begin{proof}
For $p \geq 2$ all follows from Propositions~\ref{pLpR203}, \ref{pLpR205}
and \ref{pLpR211}.
The dual of the operator $\varsigma A_2$ on $\LL_\varsigma^2$ is given 
by $\varsigma A^\#$, where $A^\#$ is the operator obtained with coefficient matrix 
equal to the transpose of the matrix $\mu$.
Hence by Proposition~\ref{pLpR211} the dual semigroup $(S^{(\varsigma,2)})^*$ 
is submarkovian and extends 
consistently to a strongly continuous semigroup on $\LL_\varsigma^p$
for all $p \in [2,\infty)$.
Then by duality the semigroup $S^{(\varsigma,2)}$ extends 
consistently to a strongly continuous semigroup on $\LL_\varsigma^p$
for all $p \in [1,2]$.
\end{proof}

\subsection{Consequences for the operators $\varsigma A_p$ on $\LL^p$.}

We have the following abstract properties for $\varsigma A_p$. 

\begin{theorem} \label{t-imagin}
Let $p \in (1,\infty)$.
Then one has the following.
\begin{statements}
\item  \label{t-imagin-1} 
The operator $\varsigma A_p$ admits a bounded holomorphic
 functional calculus on $\mathbb L^p$, with angle strictly smaller than $\frac{\pi}{2}$.
 In particular, it admits bounded imaginary powers.
\item\label{t-imagin-2} For all $\theta \in (0,1)$ one has
\[
(\varsigma A_p+1)^{-\theta}
=\frac {\sin \pi \theta}{\pi} \int_0^\infty t^{-\theta} (\varsigma A_p+1+t)^{-1} \, dt.
\]
\item  \label{t-imagin-3} If $\theta\in (0,1]$, then
$\dom \bigl ((\varsigma A_p)^\theta\bigr ) = [\mathbb L^p, \dom (\varsigma A_p)]_\theta =\dom(A_p^\theta)$, where $[\cdot,\cdot]_\theta$ denotes complex interpolation.
\end{statements}
\end{theorem}
\begin{proof}
`\ref{t-imagin-1}'.
For all $p \in [1,\infty)$ denote by $S^{(\varsigma,p)}$ the contraction semigroup on 
$\LL_\varsigma^p$ generated by $- \varsigma A_p$.
Then the semigroups $S^{(\varsigma,p)}$ with $p \in [1,\infty)$ are consistent.
Moreover, $S^{(\varsigma,2)}$ is holomorphic and bounded on a sector.
Let $p \in (1,\infty)$.
Then it follows from \cite[Proposition 3.12]{Ouh5} and duality  that 
$S^{(\varsigma,p)}$ is holomorphic and bounded in a sector (which depends on $p$).
Also $S^{(\varsigma,p)}$ is a positive contraction semigroup.
Hence the operator $S^{(\varsigma,p)}_t$ is contractively regular for all $t > 0$.
So by \cite[Proposition 2.2]{LeMX} the operator 
$\varsigma A_p$ admits a bounded holomorphic
functional calculus on $\mathbb L_\varsigma^p$, with angle strictly
 smaller than $\frac{\pi}{2}$.
This is then also the case on $\LL^p$, since $\mathbb L^p = \mathbb L_\varsigma ^p$ as 
vector spaces, with equivalent norms.

`\ref{t-imagin-2}'.
For the integral representation see \cite[(4.41)]{Lun}.

`\ref{t-imagin-3}'.
 Since $\varsigma A_p$ admits bounded imaginary powers, it follows from \cite[Theorem~4.17]{Lun} that 
\[
\dom \bigl ((\varsigma A_p)^\theta\bigr ) = [\mathbb L^p, \dom(\varsigma A_p)]_\theta
 .  \]
Since $\dom(\varsigma A_p) = \dom( A_p)$, one has 
$\dom \bigl ((\varsigma A_p)^\theta\bigr ) = [\mathbb L^p, \dom(A_p)]_\theta$, 
and the result follows.
\end{proof}

\section{Linear parabolic equations}
\label{s-parabolic}
In this section we will draw conclusions for linear parabolic equations,
which, in particular, allow to give (\ref{e-parabol})--(\ref{e-initial})
a precise meaning and afterwards to solve it. 

In the following, $J=(0,T)$ denotes a bounded interval and $X$ a Banach space.
Throughout 
 we fix the numbers 
\[
1<s<\infty
\qquad \mbox{and} \qquad \frac{1}{s} < \alpha \leq 1.
\] 
We introduce the weighted space
\[
L_\alpha^s(J; X) = \{u \colon J\to X \;:\; [t\mapsto t^{1-\alpha} u(t)] \in L^s(J;X)\},
\]
and the corresponding weighted Sobolev space
\[
W_\alpha^{1,s}(J;X) = \{u \in L_\alpha^s(J; X) :  u'\in L_\alpha^s(J; X)\},
\]
where here and below the time derivative is understood in the sense 
of $X$-valued distributions
 (see \cite[Subsection~III.1.1]{Ama2}).
These are Banach spaces when equipped with their canonical norm, respectively.
Note that $\alpha = 1$ corresponds to 
the unweighted case, i.e., $L_1^s = L^s$.
By \cite[Lemma~2.1]{PrSi} one has
$W_\alpha^{1,s}(J; X) \subset W^{1,1}(J; X)$, which implies that
 each element of $W_\alpha^{1,s}(J;X)$ has a well-defined trace at $t=0$. 

\begin{definition} \label{d-maxreg}
Let $A$ be a closed linear operator on 
$X$ with dense domain $\dom(A)$.
We say that $ A$ has \emph {maximal parabolic
 $L_\alpha^s(J; X)$-regularity},  if for all $f\in L_\alpha^s(J; X)$ there is a unique
 solution $u \in W_\alpha^{1,s}(J; X) \cap L_\alpha^s(J; \dom(A))$ of 
\[
u' +  Au = f,\qquad u(0) = 0 .
\]
We write $M\!R_\alpha^s(J;X)$ for the class of all operators on $X$ with this property.
\end{definition}

We proceed with some comments concerning maximal parabolic regularity.

\begin{enumerate} 
\item 
It is shown in \cite[Theorem 2.4]{PrSi} that $A\in M\!R_1^s(J;X)$ if and only if 
$A\in M\!R_\alpha^s(J;X)$ for all $\alpha \in (\frac{1}{s}, 1]$, i.e., maximal 
parabolic $L_\alpha^s$-regularity is independent of the weight.
(In fact, in \cite{PrSi} 
only the case $J=(0,\infty)$ is treated, but the arguments given there also apply to bounded $J$.) 
In this sense it is natural to consider the temporal weights in the context 
of parabolic problems.
\item 
If $A\in M\!R_1^{s_0}(J_0;X)$ for an interval $J_0 = (0,T_0)$, where 
$T_0 \in (0,\infty)$ and $s_0 \in (1,\infty)$, then 
$A\in M\!R_\alpha^{s}(J;X)$ for all $T \in (0,\infty)$, $s \in (1,\infty)$ and 
$\alpha \in (\frac{1}{s}, 1]$.
This is shown in \cite[Corollary~5.4 and Theorem~7.1]{dor}.
 In this spirit, we then simply say that
$A$ satisfies maximal parabolic regularity on $X$.
\item 
The notion `maximal parabolic regularity' does not depend on the concrete
norm of the Banach space. 
In other words: an operator $A$, satisfying maximal
parabolic regularity on $X$, remains to satisfy maximal
parabolic regularity if $X$
is equipped with an equivalent norm. 
\item 
If $A$ satisfies maximal parabolic regularity on $X$, then $-A$ generates
 an analytic $C_0$-semigroup (cf.\ \cite[Corollary~4.4]{dor}).
If $X$ is a Hilbert space, then the 
converse is also true, cf.\ \cite{DeS}.
\end{enumerate}

For the case of nontrivial initial values, the following has been proved in  
\cite[Theorem~3.2]{PrSi}.
We denote by $(\cdot,\cdot)_{\theta,s}$ the real interpolation functor, 
cf.\ \cite[Sections~1.3 and 1.6]{Tri}.

\begin{proposition} \label{ps}
Suppose that $A$ satisfies maximal parabolic regularity on $X$.
Then for all $f\in L_\alpha^s(J;X)$ 
and $u_0\in (X,\dom(A))_{\alpha-\frac{1}{s},s}$ the Cauchy problem
\[
u' +  Au = f,\qquad u(0) = u_0,
\]
has a unique solution $u\in W_\alpha^{1,s}(J; X) \cap L_\alpha^s(J; \dom(A))$, and the estimate
\begin{equation}\label{cont-dep}
\|u'\|_{L_\alpha^s(J; X)} +\|u\|_{L_\alpha^s(J; \dom(A))}
\le c\big( \|u_0\|_{(X,\dom(A))_{\alpha-\frac{1}{s},s}} + \|f\|_{L_\alpha^s(J; X)}\big)
\end{equation}
is valid for some constant $c$, independent of $f$ and $u_0$.
\end{proposition}

By working in temporally weighted spaces one can thus reduce the regularity of the 
initial values $u_0$ almost up to the base space $X$.

We have the following embeddings for the weighted maximal regularity class.
The space of $\gamma$-H\"older continuous functions is denoted by $C^\gamma$.

\begin{proposition}\label{extra-reg} If $A$ satisfies maximal parabolic regularity on $X$, then
\[
W_\alpha^{1,s}(J; X) \cap L_\alpha^s(J; \dom(A)) 
\subset
BU\!C(\overline{J}; (X, \dom(A))_{\alpha-\frac{1}{s},s}) \cap C(J; (X, \dom(A))_{1-\frac{1}{s},s}).
\]
Moreover, for every $\theta \in [0,\alpha-\frac {1}{s})$ there is a $\gamma \in (0,1)$ such
that 
\[
W_\alpha^{1,s}(J; X) \cap L_\alpha^s(J; \dom(A)) \subset C^\gamma(\overline{J};[X,\dom(A)]_\theta).
\]
\end{proposition}
\begin{proof} The first inclusion is shown in \cite[Proposition~3.1]{PrSi}.
The second one can be proved along the lines of \cite[Lemma~1]{DMRT}.
\end{proof}

We apply a classical result of Lamberton \cite{Lamb} to the operators $\varsigma A_p$.

\begin{theorem} \label{t-qreg}
Let $\varsigma \colon \Omega \cup \Gamma \to (0,\infty)$ be a measurable 
function such that $\varsigma,\varsigma^{-1} \in \LL^\infty$. 
Then for all $p\in(1,\infty)$ 
the operator $\varsigma A_p$ satisfies maximal parabolic regularity on $\mathbb L^p$.
\end{theorem}
\begin{proof}
Theorem~\ref{t-mult03-new} states that the semigroup generated by $-\varsigma A_2$ on 
$\mathbb L^2_\varsigma$ is bounded and analytic, and that it 
extents consistently to a contractive semigroup on $\mathbb L^q_\varsigma$ 
for all $q\in [1,\infty]$.
Now the result is a consequence of \cite[Corollary.~1.1]{Lamb}.
\end{proof}

In order to include lower order terms into the boundary and interface conditions we need some preparation.

\begin{proposition} \label{pLpR301}
Let $p \in (1,\infty)$ and $\theta \in (0,1)$ be such that 
$d-1 < \theta \, p$. 
Then one has
$\dom\bigl( (\varsigma A_p)^\theta \bigl) \subset \LL^\infty$.
\end{proposition}
\begin{proof} 
Since $\dom \bigl((\varsigma A_p)^\theta \bigl) = \dom \bigl( (A_p+1)^\theta  \bigl)$ by  
Theorem~\ref{t-imagin}\ref{t-imagin-3} and \cite[Lemma 4.1.11]{Lun}, it suffices to show that $(A_p +1)^{-\theta}$ 
maps $\LL^p$ into $\LL^\infty$. In \cite[Section 2.6]{Paz} it is shown that
\[
(A_p +1)^{-\theta} 
= \frac{1}{\Gamma(\theta)} \int_0^\infty t^{\theta-1} \, e^{-t} \, S_t \, dt
 .  \]
Now the assertion follows from the estimate of Lemma~\ref{lLpR206}\ref{lLpR206-1}.
\end{proof}

\begin{corollary} \label{c-gebropot}
Suppose $p \in (\frac{d-1}{\alpha - \frac{1}{s}}, \infty)$.
Then
$(\mathbb L^p,\dom (\varsigma A_p))_{\alpha-\frac {1}{s},s}$ continuously embeds into 
$ \mathbb L^\infty$.
\end{corollary}
\begin{proof}
Fix $\theta \in (\frac{d-1}{p}, \alpha-\frac {1}{s})$.
Then Proposition~\ref{pLpR301} yields 
$\dom\bigl (( \varsigma A_p)^\theta \bigr ) \subset  \mathbb L^\infty$.
But
\[
(\mathbb L^p,\dom(\varsigma A_p))_{\alpha-\frac {1}{s},s} \subset
 (\mathbb L^p,\dom(\varsigma A_p))_{\theta,1}
\subset [ \mathbb L^p, \dom(\varsigma A_p)]_\theta
\]
by \cite[Propositions 1.1.3, 1.3.2 and Corollary 2.1.8]{Lun}, and the latter interpolation space 
equals $\dom\bigl ((A_p)^\theta\bigr )$ by Theorem \ref{t-imagin}\ref{t-imagin-3}.
\end{proof}

\begin{definition} \label{d-emb}
Fix $b \in L^p(\Gamma \cup \SX;d\rho)$.
Define the operator 
$B \colon  \mathbb L^\infty \to \mathbb L^p $ by 
\[
B(f_\Omega,f_\partial) = (0, b \, f_\partial)
\]
for all $f = (f_\Omega,f_\partial) \in L^p(\Omega) \oplus L^p(\Gamma\cup \SX;d\rho) \cong \LL^p$.
\end{definition}

Note that $b$ is allowed to be complex valued.

\begin{theorem} \label{t-haupt}
Let $p \in (d-1,\infty)$. 
Then the operator $\varsigma (A_p+B)$ satisfies maximal parabolic regularity on $\mathbb L^p$.
\end{theorem}
\begin{proof}
One deduces from Corollary \ref{c-gebropot} that the operator
$\varsigma B$ acts continuously on an interpolation space
between $\dom(\varsigma A_p)$ and $\mathbb L^p$.
Then the result follows from the 
perturbation theorem \cite[Theorem~6.2]{dor}.
\end{proof}

\begin{rem} \label{r-expltime}
In a somewhat more general concept $B$ may also depend explicitly on time, see
\cite{ACFP}.
\end{rem}

Now we are in the position to solve the parabolic problem
(\ref{e-parabol})--(\ref{e-initial}) in terms of the realization of the operator $A_p$.

\begin{theorem} \label{t-solution} 
Let $T \in (0,\infty)$ and set $J=(0,T)$.
Let $ p \in (d-1,\infty)$, $s \in (1,\infty)$ and $\alpha \in (\frac{1}{s},1]$.
Let $\Omega$ be a bounded domain in $\R^d$ with $d >1$, let
$\Gamma$ be an open part of its boundary $\partial\Omega$ and $\SX \subset \Omega$.
Adopt the Assumptions~{\rm \ref{a-region}}, {\rm \ref{a-d-1}} and {\rm \ref{a-coeff}}.
Let $\varepsilon \in \mathbb L^\infty$  be a positive function with
a positive essential lower bound and let $b$ as in
 Definition~{\rm \ref{d-emb}}.
Then the initial value problem \eqref{e-parabol}--\eqref{e-initial} 
admits a solution in the following sense:
for all $f \in L_\alpha^s(J;\mathbb L^p)$ and $u_0 \in 
(\mathbb L^p,\dom(A_p))_{\alpha-\frac {1}{s},s}$ there is a unique function
 $u \in W_\alpha^{1,s}(J;\mathbb L^p) \cap L_\alpha^s(J;\dom(A_p))$ satisfying
\begin{equation} \label{e-euat}
\varepsilon u' +A_p u +Bu=f, \qquad u(0)=u_0.
\end{equation}
\end{theorem}
\begin{proof}
One reformulates \eqref{e-euat} as
\[
u' +\varepsilon^{-1} A_p u +\varepsilon^{-1}Bu=\varepsilon^{-1}f, \qquad u(0)=u_0.
\]
Obviously, $\varepsilon^{-1}f$ satisfies the same assumptions as $f$.
Moreover, one has
$\dom(A_p)=\dom(\varepsilon^{-1}A_p) = \dom(\varepsilon^{-1} (A_p + B))$, with equivalent norms.
This implies that 
\[
(\mathbb L^p,\dom(A_p))_{\alpha-\frac {1}{s},s}
=(\mathbb L^p,\dom(\varepsilon^{-1}(A_p + B)))_{\alpha-\frac {1}{s},s}
 .  \]
The assertion then follows from Proposition \ref{ps} and Theorem \ref{t-haupt}.
\end{proof}

\begin{rem} In the situation of the theorem, the solution depends continuously 
on the data, due to (\ref{cont-dep}).
 Proposition \ref{extra-reg} gives further 
regularity properties of a solution. 
Moreover, again by (\ref{cont-dep}), it is straightforward to see that the solution depends continuously 
on the function $\varepsilon$, with respect to the $\mathbb L^\infty$-norm.
\end{rem}

\begin{rem} \label{r-reell}
Since the coefficient function $\mu$ is real valued, the resolvent of 
$\varsigma A_p$ 
commutes with complex conjugation on the spaces $\mathbb L^p$. The latter is
also true for the semigroup operators $e^{-t\varsigma A_p}$. Thus, the restriction of
$\varsigma A_p$ to real spaces $\mathbb L^p_\R$ also satisfies maximal parabolic regularity.
If $B$ is induced by a real valued function, then the same is  true for the 
operator $\varsigma (A_p+B)$.
\end{rem}

\begin{rem} \label{rem-heuristics-2} At the end of this section, let us give more detailed, partly
 heuristic arguments what the real advantage is of the treatment of our parabolic equations
 in the spaces $\mathbb L^p$.

When considering the solution 
$u$ of a parabolic equation $u'+Au =f$ on a Banach space $X$
one can form the dual pairing with any $\psi\in X^*$ to obtain
\begin{equation} \label{e-dualpaireq}
\frac {\partial }{\partial t} \langle u,\psi \rangle + \langle Au,\psi\rangle 
= \langle f,\psi \rangle. 
\end{equation}
E.g.,  if $X=W^{-1,2}(\Omega)$, then one can choose for $\psi$ as any element of 
$W^{1,2}_0(\Omega)$, but \emph{not} an indicator function 
of a subset of $\Omega$.
In our setting, the situation is different: if $X=\mathbb L^p$, then the dual pairing 
with the indicator function $\chi_U$ of a measurable set $U\subset \Omega$ is admissible.
Then \eqref{e-dualpaireq} reads, there $A$ taken as the $\mathbb L^2$-realization of
$A_2$,
\begin{equation} \label{e-balance}
\frac {\partial }{\partial t} \int_U u \,(d x +d\rho) +\int_U
(A_2u) \,(d x +d\rho)
=\int_U f \,(dx +d\rho).
\end{equation}
Since $A_2u \in \mathbb L^2$ for almost every time point $t$ we are now at least in principle in the position to 
rewrite $\int_U (A_2u) \,(d x +d\rho)$ as a boundary integral and thus to recover 
from \eqref{e-balance} the `original' physical balance law for  \eqref{e-parabol}--\eqref{e-initial}.

Indeed, applying \eqref{e-opdef} with  $v\in C_c^\infty(\Omega)$, it follows that the
 distributional divergence of $\mu \nabla u $ is given by the finite Radon measure 
induced by $(A_2u|_\Omega, A_2u|_{\Sigma}) \in L^2(\Omega)\times L^2(\Sigma; d \mathcal H_{d-1})$
 with respect to $dx + d\mathcal H_{d-1}$ (see also Remark \ref{r-hypersurf}). Under 
certain further assumptions on $\mu \nabla u$ or $U$ one can apply the generalized
 Gauss-Green theorems of e.g. \cite{CTZ}, \cite{Fug} and \cite{Zie1} to obtain 
\begin{equation}\label{Gauss-Green}
 \int_U (A_2u) \,(d x +d\rho) = \int_{\partial U} \nu\cdot \mu \nabla u  \, d \mathcal H_{d-1},
\end{equation}
where $\nu\cdot  \mu \nabla u \in L^1( \partial U; d \mathcal H_{d-1})$ is 
`the generalized normal component of the corresponding flux', see ibidem. 

Substituting \eqref{Gauss-Green} in \eqref{e-balance} gives the desired balance law, 
as is classical when $\nabla \cdot \mu \nabla u$ is 
an $L^2(\Omega)$-function; compare \cite[Chapter~21]{Somm} and \cite{CLolas}.
As already mentioned in the introduction, this is the basis for local flux balances, 
which are crucial for the foundation of Finite Volume methods for the numerical solution
of such problems, compare \cite{BRF}, \cite{FuhL} and \cite{gartn}. 
\end{rem}

\section{Quasilinear parabolic equations}
\label{s-semilinear}

In this section we treat a nondegenerate quasilinear variant of 
\eqref{e-parabol}--\eqref{e-initial}, 
including nonlinear terms in the dynamic equations on $\Gamma$ and $\SX$, i.e., 
\begin{alignat}{2}
\varepsilon \partial_t \mathfrak b(u)-\nabla \cdot \mu 
\mathfrak a(u)\nabla u & = F_\Omega(t,u)  
& \qquad & \text{in }\,J\times (\Omega\setminus \SX), \label{quasi-1} \\
u & =  0  & &\text {on }\,  J \times (\partial \Omega \setminus \Gamma), \label{quasi-2}\\
\varepsilon \partial_t \mathfrak b(u) +\nu \cdot \mu 
\mathfrak a(u)\nabla u & =  F_\Gamma(t,u) &  & \text{on }\,
 J \times \Gamma, \label{quasi-3}\\
\varepsilon \partial_t \mathfrak b(u) +[\nu_\SX \cdot \mu \mathfrak a(u)
\nabla u]  & =  F_\SX(t,u)  &  & \text{on }\, J \times \SX, \label{quasi-4} \\
u(T_0) & =   u_0  & &\text{in }\, \Omega \cup \Gamma,\label{quasi-5}
\end{alignat}
where  $J=(T_0,T_1)$ is a bounded interval.
Interesting examples for the nonlinearities on the
 left-hand side are e.g.\ when  $\mathfrak b$ and $\mathfrak a$ are an exponential, or the
Fermi--Dirac distribution function $\mathcal F_{1/2}$, which is given by
\[ \mathcal F_{1/2}(s) 
:= \frac{2}{\sqrt{\pi}} \, \int_0^\infty \frac{\sqrt{\xi}}{1 + e^{\xi- s}} \, d  \xi.
\]
Further, in phase separation problems a rigorous formulation as a minimal problem for the free
energy reveals that $\mathfrak {a} = \mathfrak {b}^\prime$ is appropriate.
This topic has been thoroughly investigated in \cite{Qua},
\cite{QRV}, \cite{GiacL1}, and \cite{GiacL2}, see also
\cite{GajS} and \cite{Grie3}.

\medskip

We consider from now on the real part $\LL_\R^p$ of the spaces $\mathbb L^p$
and the operators $A_p$.
For simplicity we write $\mathbb L^p$ for $\LL_\R^p$.
As in the linear case we give the quasilinear equation a suitable functional analytic 
formulation, and within this framework the problem will then be solved (see Definition 
\ref{quasi-solution} and Theorem \ref{t-semilinear} below).
Again throughout this section we  fix the numbers
\[
1<s<\infty
\qquad \mbox{and} \qquad \frac{1}{s} < \alpha \leq 1.
\] 
We impose the following conditions on the coefficients on the left-hand side
of (\ref{quasi-1})--(\ref{quasi-5}).

\begin{assu} \label{a-Verteil}
The coefficient matrix $\mu$ is real-valued, $\mathfrak b \in W^{2,\infty}_{\text{loc}}(\R)$ is
 such that $\mathfrak b'$ is positive, and $\mathfrak a \in W^{1,\infty}_{\text{loc}}(\R)$  is
 positive and satisfies 
$\int_0^\infty\mathfrak a(\zeta)\, d\zeta = \infty 
= \int_{-\infty}^0\mathfrak a(\zeta)\, d\zeta$. 
\end{assu}

Note that we do not require monotonicity for $\mathfrak a$.
In particular, terms of the form
 $\mathfrak a(u) = \eta + |u|^m$ with $\eta>0$ and $m\geq 1$ can be treated, 
that arise e.g.\ 
as a regularization of the porous medium equation.

It is in general not to expect that the domain of the realization of  
$-\nabla \cdot \mu \mathfrak a(v) \nabla$ on $\mathbb L^p$ as in 
Section \ref{SLpS2.2} is independent
 of $v\in L^\infty(\Omega)$.
Consider, e.g., the case of a smooth geometry with $\mu \mathfrak a(v)$ 
equal to a constant on the one hand and a nonsmooth $\mu \mathfrak a (v)$ on the other hand. 
  This observation motivates our definition of a solution of \eqref{quasi-1}--\eqref{quasi-5},
 which we describe in the following.
We put
\[
\mathfrak K (\xi):= \begin{cases}
\int_0^\xi \mathfrak a(\zeta) \,d\zeta, \;\text {if} \; \xi \ge 0, \\
-\int_\xi^0 \mathfrak a(\zeta) \,d\zeta, \;\text {if} \; \xi < 0.
\end{cases}
\] 
Then the assumptions on $\mathfrak a$ imply that
\[
\mathfrak K \colon \R\to \R \text{ is bijective},\quad 
\mathfrak K, \mathfrak K^{-1} \in W^{1,\infty}_{\text{loc}}(\R), \quad 
\mathfrak K'=\mathfrak a, \quad 
\mbox{and} \quad
\mathfrak K(0)=0=\mathfrak K^{-1}(0).
\]
In the sequel we identify the functions $\mathfrak b, \mathfrak K, \mathfrak K^{-1}$
 with the Nemytzkii operators they induce.
The reformulation of \eqref{quasi-1}--\eqref{quasi-5}
 is based on the so-called Kirchhoff transform $w = \mathfrak K(u)$.
This (formally) gives $\mathfrak a(u) \nabla u = \nabla w$
 and $\partial_t (\mathfrak b(u)) = \frac{\mathfrak b'}{\mathfrak a} \, (\mathfrak K^{-1}(w)) \partial_t  w$.
Since $\mathfrak K(0) = 0$, the problem \eqref{quasi-1}--\eqref{quasi-5} 
thus transforms into 
\begin{alignat*}{2}
\partial_t w - \eta(w) \nabla \cdot \mu \nabla w
& =  \eta(w) F_\Omega(t,\mathfrak K^{-1}(w))  
& \qquad & \text{in }\,J\times  (\Omega\setminus \SX), \\
w & =  0  & &\text {on }\,  J \times (\partial \Omega \setminus \Gamma), \\
 \partial_t w + \eta(w) \nu \cdot \mu \nabla w  & =   \eta(w)  F_\Gamma
(t,\mathfrak K^{-1}(w)) &  & \text{on }\, J \times \Gamma, \\
\partial_t w + \eta(w)[\nu_\SX \cdot \mu 
\nabla w]  & =   \eta(w)  F_\SX(t,\mathfrak K^{-1}(w))  &  & \text{on }\,
 J \times \SX, \\
w(T_0) & =   \mathfrak K(u_0)  & &\text{in }\, \Omega \cup \Gamma,
\end{alignat*}
where we have set 
\[
\eta(w) := \varepsilon ^{-1} \, \frac{\mathfrak a}{\mathfrak b'} \, \mathfrak K^{-1}(w).
\]
For all $t\in J$, let us further define the operator
\begin{equation}\label{quasi-R}
 R(t,w) := \begin{cases} 
\eta(w|_{\Omega})F_\Omega (t,\mathfrak K^{-1}(w|_{\Omega}))\; \text{ on} \;\Omega \setminus \SX,\\
\eta(w|_{\Gamma})  F_\Gamma(t,\mathfrak K^{-1}(w|_{\Gamma})) \; \text{ on} \;\Gamma, \\
\eta(w|_{\SX}) F_\SX (t,\mathfrak K^{-1}(w|_{\SX})) \; \text{ on} \;\SX,
\end{cases}
\end{equation}
acting on real-valued functions defined on $\Omega\cup \Gamma$.

\begin{definition} \label{quasi-solution} Let $p \in (\frac{d-1}{\alpha - \frac{1}{s}}, \infty)$,
 and  let $A_p$ be the realization of $-\nabla \cdot \mu \nabla$ on $\mathbb L^p$ as in 
Section \ref{SLpS2.2}.
We say that $u\in C([T_0,T_1];\mathbb L^\infty)$ is a solution of
 \eqref{quasi-1}--\eqref{quasi-5} on $J$ if 
\[
\mathfrak K (u) \in W_\alpha^{1,s}(J; \mathbb L^p)\cap L_\alpha^s(J; \dom(A_p)),
\]
and if $w = \mathfrak K(u)$ satisfies
\begin{equation}\label{quasi-abstract}
\partial_t w + \eta(w) A_p w = R(\cdot,w) \quad \text{on }J, \qquad w(T_0) = \mathfrak K (u_0).
\end{equation}
\end{definition}

If $\mathfrak K(u)$ is as above, then $u\in C([T_0,T_1];\mathbb L^\infty)$ is already a consequence of
 Proposition \ref{extra-reg}, Corollary \ref{c-gebropot} and the regularity of $\mathfrak K$.
 Proposition \ref{extra-reg} shows that in fact  $u\in C^\gamma([T_0,T_1];\mathbb L^\infty)$
 for some $\gamma>0$.
For specific choices of $\mathfrak K$ additional regularity may carry 
over from $\mathfrak K(u)$ to $u$.
In any case one has $u(t,\cdot) \to u_0$ as $t\to T_0$ in the $\mathbb L^\infty$-norm.

Observe further that in the definition it is necessary that 
$\mathfrak K (u_0) \in (\mathbb L^p, \dom(A_p))_{\alpha-\frac{1}{s},s}$.
It would be interesting to
 find another description for this condition for a class of nonlinearities $\mathfrak a$.
 If $\mathfrak a$ is constant, then  a solution in the above sense can be defined for all 
$u_0 \in (\mathbb L^p, \dom(A_p))_{\alpha-\frac{1}{s},s}$.

\medskip 

If $\mathfrak a = \mathfrak b'$, then \eqref{quasi-abstract} is in fact a semilinear problem.
This
 is in particular the case for the phase separation problems from above. 

\medskip

To solve \eqref{quasi-abstract} we intend to use the following abstract existence and uniqueness result, 
which is proved in \cite{pru2} for the temporally unweighted case $\alpha = 1$.
The proof in \cite{pru2} literally carries over to the weighted case $\alpha < 1$.

\begin{proposition} \label{p-pruess}
Let $X,D$ be Banach spaces such that $D$ embeds continuously and densely into $X$. Assume
$\mathcal{A} \colon (X, D)_{\alpha -\frac {1}{s},s} \to \mathcal{L}(D, X)$ and 
$\mathcal R  \colon  J \times (X, D)_{\alpha-\frac {1}{s},s} \to X$
are such that $\mathcal R(\cdot,w_0)$ is measurable for all $w_0\in (X, D)_{\alpha-\frac {1}{s},s}$, 
that $ \mathcal R(\cdot, 0) \in L_\alpha^s(J;X)$ and that for  all  $M > 0$ there are
$C_M > 0$ and $r_M \in L_\alpha^s(J)$ with
\[
\| \mathcal A(w_1) - \mathcal A(w_2) \|_{\mathcal L( D,X)} 
\le C_M \, \| w_1 - w_2\|_{(X,D)_{\alpha-\frac {1}{s},s}} 
\]
and
\[
\| \mathcal R(t,w_1) - \mathcal R(t,w_2)\|_X \le r_M(t) \, \| w_1 - w_2 \|_{(X, D)_{\alpha-\frac {1}{s},s}}
\qquad \mbox{ for a.e.\ } t \in J,
\]
for all $w_1,w_2 \in (X, D)_{\alpha-\frac {1}{s},s}$ with 
$\|w_1\|_{(X, D)_{\alpha-\frac {1}{s},s}} \le M$ and 
$\|w_2\|_{(X, D)_{\alpha-\frac {1}{s},s}} \le M$. Assume further that for any $w_0\in (X, D)_{\alpha -\frac {1}{s},s}$ the operator $\mathcal A(w_0)$ with domain $D$ on $X$ satisfies maximal parabolic regularity.

Then for all $w_0 \in (X, D)_{\alpha-\frac {1}{s},s}$ there are 
$T^*\in (T_0,T_1]$ and a unique maximal solution $w$ of
\[
w'+\mathcal A(w)w=  \mathcal R(\cdot,w) \quad \text{on }(T_0,T^*), \qquad w(T_0)=w_0,
\]
such that 
$w \in W_\alpha^{1,s}(T_0,T;X) \cap L_\alpha^s(T_0,T; D)$
for all $T\in (T_0,T^*)$.
\end{proposition}

We apply this result to \eqref{quasi-abstract}. 
Suppose  $\mathfrak b$ and $\mathfrak a$ satisfy Assumption \ref{a-Verteil}.
Let $p \in (\frac{d-1}{\alpha - \frac{1}{s}}, \infty)$, 
$X= \mathbb L^p$, $D = \dom (A_p)$ and $\mathcal A(w) = \eta(w) A_p$ for all 
$w \in (\mathbb L^p, \dom(A_p))_{\alpha- \frac{1}{s},s}$. 
Corollary \ref{c-gebropot} implies that
\begin{equation}\label{quasi-embed}
 (\mathbb L^p, \dom(A_p))_{\alpha- \frac{1}{s},s} \subset \mathbb L^\infty.
\end{equation}
Thus if $w_0 \in (\mathbb L^p, \dom(A_p))_{\alpha- \frac{1}{s},s}$ and
$\|w_0\|_{(\mathbb L^p, \dom(A_p))_{\alpha- \frac{1}{s},s}} \leq M$ for a given number $M$, 
then it follows from \eqref{quasi-embed} that the image of $\Omega\cup \Gamma$ under $w_0$ is 
almost everywhere contained in a compact interval that only depends on $M$. 
In particular, this gives $\eta(w_0), \eta(w_0)^{-1} \in \mathbb L^\infty$, 
and the operator $\mathcal A(w_0)$ with domain $\dom(A_p)$ on $\mathbb L^p$ satisfies maximal parabolic regularity 
by Theorem \ref{t-qreg}. 

The function $\eta$ is locally Lipschitz continuous on $\R$.
Therefore
\begin{align*}
\|\mathcal A(w_1) - \mathcal A(w_2)\|_{\mathcal L(\dom(A_p),\mathbb L^p)}&\,  \leq \|\eta(w_1) - \eta(w_2)\|_{\mathbb L^\infty} \\
&\, \leq C_M\| w_1-w_2\|_{\mathbb L^\infty} \leq \| w_1-w_2\|_{(\mathbb L^p, \dom(A_p))_{\alpha- \frac{1}{s},s}}
\end{align*}
for all $w_1,w_2 \in (\LL^p,\dom(A_p))_{\alpha-\frac {1}{s},s}$ with 
$\|w_j\|_{(\LL^p,\dom(A_p))_{\alpha-\frac {1}{s},s}} \le M$ for all $j \in \{ 1,2 \} $.
This verifies the conditions of the above proposition for $\mathcal A$.

We next present sufficient conditions for the functions $F_\Omega$, $F_\Gamma$ and $F_\SX$
such that the operator $R$, defined in \eqref{quasi-R}, satisfies the  
conditions for $\mathcal R$ in Proposition~\ref{p-pruess}.

\begin{assu} \label{rhs-assu} For all $\xi \in \R$ the mappings 
$F_\Omega(\cdot,\xi) \colon J\to \R$, 
$F_\Gamma(\cdot,\xi) \colon J\to \R$ and $F_\SX(\cdot,\xi) \colon J\to \R$
are measurable.
For all $M>0$ there is  $r_M \in L_\alpha^s(J)$ such that
\[
|F_\Omega(t,\xi_1) - F_\Omega(t,\xi_2)| \le r_M(t) \, | \xi_1 -\xi_2 | 
\]
for a.e.\ $t \in J$ and $\xi_1,\xi_2 \in \R$ with $|\xi_1|, |\xi_2| \leq M$;
and analogous conditions for $F_\Gamma$ and $F_\SX$.
\end{assu}

Under the above assumption, \eqref{quasi-embed} implies that 
$R(\cdot,w_0) \colon J \to \mathbb L^p$ is measurable for all $w_0\in (\mathbb L^p,
 \dom(A_p))_{\alpha-\frac{1}{s},s}$ and that $R(\cdot,0)\in L_\alpha^s(J)$.
We verify 
the Lipschitz property for the first component of $R$.
If $M > 0$, and $w_1,w_2 \in (\mathbb L^p, \dom(A_p))_{\alpha-\frac{1}{s},s}$ with 
$\|w_1\|_{(\mathbb L^p, \dom(A_p))_{\alpha-\frac{1}{s},s}} \leq M$ and
$\|w_2\|_{(\mathbb L^p, \dom(A_p))_{\alpha-\frac{1}{s},s}} \leq M$,
then for a.e.\ $t\in J$ we have 
\begin{align}\label{quasi-R-est}
\|\eta(w_1|_\Omega) F_\Omega(t,\mathfrak K^{-1}&\,(w_1|_\Omega)) -  \eta(w_2|_\Omega)
F_\Omega
(t,\mathfrak K^{-1}(w_2|_\Omega))\|_{L^p(\Omega)}\\
&\,  \leq \|\eta(w_1|_\Omega) - \eta(w_2|_\Omega)\|_{L^\infty(\Omega)} \|
F_\Omega(t,\mathfrak K^{-1}(w_1|_\Omega))\|_{L^p(\Omega)}
 \nonumber \\
&\, \qquad + \|\eta(w_2|_\Omega)\|_{L^\infty(\Omega)} \|F_\Omega(t, \mathfrak K^{-1}(w_1|_\Omega)) - 
F_\Omega(t, \mathfrak K^{-1}(w_1|_\Omega))\|_{L^p(\Omega)} \nonumber\\
&\, \leq C_M \big( \|w_1|_\Omega - w_2|_\Omega\|_{L^\infty(\Omega)} + \widetilde{r}_{M}(t)
 \|\mathfrak K^{-1}(w_1|_\Omega) - \mathfrak K^{-1}(w_2|_\Omega)\|_{L^p(\Omega)}\big) \nonumber\\
&\,\leq C_M(1+ \widetilde{r}_{M}(t))\|w_1 - w_2\|_{(\mathbb L^p, \dom(A_p))_{\alpha-\frac{1}{s},s}}\nonumber,
\end{align}
for a suitable $\widetilde{r}_{M}\in L_\alpha^s(J)$.
The same arguments apply to the other components of $R$, and thus $R$ is as desired to apply the proposition. 

We have proven the main result of this section.

\begin{theorem} \label{t-semilinear}
Let $p \in (\frac{d-1}{\alpha - \frac{1}{s}}, \infty)$, and suppose that 
$\Omega$, $\Gamma$, $\SX$, and  $\varepsilon$ are as in 
Theorem~{\rm \ref{t-solution}}, that $\mu$, $\mathfrak b$ and $\mathfrak a$ are as in 
Assumption~{\rm \ref{a-Verteil}}, and that $f$, $g$ and $h$ are as in 
Assumption~{\rm \ref{rhs-assu}}. 
 Then for all $u_0\in \mathbb L^\infty$ with 
$\mathfrak K(u_0) \in  (\mathbb L^p,\dom(A_p))_{\alpha-\frac {1}{s},s}$ there are
$T^*= T^*(u_0) \in (T_0,T_1]$ and a 
unique maximal solution  $u \in C([T_0,T^*); \mathbb L^\infty)$ of 
\eqref{quasi-1}--\eqref{quasi-5} 
in the sense of Definition~{\rm \ref{quasi-solution}}. 
This means that for all $T_0 < T < T^*$ we have
\[
\mathfrak K (u) \in W_\alpha^{1,s}(T_0,T; \mathbb L^p)\cap L_\alpha^s(T_0,T; \dom(A_p)),
\]
and $\mathfrak K(u)$ is the unique solution of 
\begin{equation} \label{quasi-eq}
 \partial_t w + \eta(w) A_p w = R(\cdot,w) \quad \text{on }(T_0,T), \qquad w(T_0) = \mathfrak K (u_0).
\end{equation}
\end{theorem}

Instead of $F_\Omega$, $F_\Gamma$ and $F_\Sigma$ one can easily find 
also non-local maps such that the corresponding operator $R$ satisfies the condition of 
Proposition \ref{p-pruess}.
One can take for example (linear or nonlinear) integral operators with
 suitable kernel properties.
Moreover, in our example, $F_\Omega$ maps $L^\infty(\Omega)$
into itself, while $F_\Gamma$ maps $L^\infty(\Gamma)$ itself, and correspondingly also for $F_\SX$, 
i.e., the mapping $R$ has no crossing terms.
This is also not necessary in general.

The nonlinearity in the elliptic operator may also be a nonlocal operator.
This case 
arises e.g.\ in models for the diffusion of bacteria; 
see \cite{CC}, \cite{CLovat} and references therein. 

\medskip

We end this section with some comments on the case when \eqref{quasi-1}--\eqref{quasi-5} is semilinear,
 i.e., when $\mathfrak b = \mathfrak K = \text{id}$, such that $u$ itself solves the realization 
\eqref{quasi-eq} of the problem.

The following is a useful criterion for the global existence of solutions.

\begin{proposition} 
Adopt the assumptions of Theorem~{\rm \ref{t-semilinear}}.
Suppose in addition 
that $\mathfrak b = \mathfrak K = \text{id}$, and let $u\in C([T_0,T^*); \mathbb L^\infty)$ 
be the maximal solution of \eqref{quasi-1}--\eqref{quasi-5}.
If
\[
\limsup_{t\to T^*}\|u(t,\cdot)\|_{L^\infty(\Omega)} < \infty,\]
then $u$ is a global solution, i.e., $T^* = T_1$ and $u\in W_\alpha^{1,s}(J; \mathbb L^p)\cap L_\alpha^s(J; \dom(A_p))$.
\end{proposition}
\begin{proof}
By Proposition \ref{ps}, for all $T< T^*$ the solution $u$ satisfies 
\begin{equation}\label{maxreg-nl}
 \|u'\|_{L_\alpha^s(T_0,T; \mathbb L^p)} +\|u\|_{L_\alpha^s(T_0,T; \dom(A_p))}
\le c\big( \|u_0\|_{(X,\dom(A_p))_{\alpha-\frac{1}{s},s}} 
    + \|R(\cdot, u)\|_{L_\alpha^s(T_0,T; \mathbb L^p)}\big),
\end{equation}
where $c$ is uniform in~$T$.
Observe that
$\|u(t,\cdot)\|_{\mathbb L^\infty}\leq \|u(t,\cdot)\|_{ L^\infty(\Omega)}$ for 
almost every $t$ by the definition of the trace (see Section \ref{s-function-spaces}). 
Hence $M = \|u\|_{L^\infty(T_0,T^*; \mathbb L^\infty)}< \infty$. 
Estimates as in \eqref{quasi-R-est} yield
\begin{align*}
 \|R(\cdot, u)\|_{L_\alpha^s(T_0,T^*; \mathbb L^p)}  
&\, \leq \|R(\cdot,0)\|_{L_\alpha^s(T_0,T^*; \mathbb L^p)} 
    +  C_M\big(1+ \|\widetilde{r}_{M}\|_{L_\alpha^s(T_0,T^*)}\big).
\end{align*}
Therefore the terms on the left-hand side of (\ref{maxreg-nl}) are bounded uniformly in~$T$.
By \cite[Corollary 3.2]{pru2}, this implies $T^* = T_1$.\end{proof}

We finally comment on the asymptotics of solutions.

\begin{rem}
 Under the assumptions of Theorem \ref{t-semilinear}, in the autonomous semilinear case the solutions form a 
local semiflow in the phase space $\dom(A_p^\theta)$, where $\theta$ is sufficiently close to $1$.
Since the resolvent of $A_p$ is compact by 
Lemma~\ref{lLpR206}\ref{lLpR206-3}, the solution semiflow is compact, and 
bounded orbits are relatively compact.
This property is very useful in studying the long-time behaviour of solutions.
\end{rem}

\section{Concluding remarks}
\label{s-conclud}

\begin{rem} \label{c-additional term}
The realization of \eqref{e-parabol}--\eqref{e-initial} in Section \ref{s-parabolic} still 
enjoys maximal regularity if one adds a term $bu$ in the dynamic equation on $J \times \SX$ 
and imposes suitable conditions on $b$.
\end{rem}

\begin{rem} \label{c-couple}
Everything can be done also for systems which couple in the reactions.
\end{rem}
\begin{rem} \label{r-nonaut}
The fundamental result of Pr\"uss (Proposition \ref{p-pruess}) allows to treat the quasilinear
 problem \eqref{quasi-1}--\eqref{quasi-5} also in the case where the nonlinearities 
$\mathfrak b$ and $\mathfrak a$ depend explicitly on time.
We did not carry out this here for the
sake of technical simplicity.
\end{rem}
\begin{rem} \label{c-Hoel}
If one requires $\Omega$ to be a Lipschitz domain and, additionally, imposes
a certain compatibility condition between $\Gamma$ and its complement in the 
boundary (see \cite{Groe}, \cite{HMRS}), then $(-\nabla \cdot \mu \nabla
+1)^{-1}$ maps $\breve W^{-1,q}_\Gamma$, i.e., the anti-dual space of $W^{1,q}_\Gamma$, 
into a H\"older space, if $q >d$.
If $s$ in Theorem \ref{t-solution}/Theorem \ref{t-semilinear}
 is chosen sufficiently large, then the corresponding solutions are even H\"older continuous in space and 
time, compare \cite{DMRT}.
\end{rem}

\begin{rem} \label{c-move}
What cannot be treated within this framework is the case where $\SX$ moves in 
$\Omega$ in time.
If one wants to include this, the concept of \cite{HaR}
should be adequate, see also \cite{HaR3}.
\end{rem}
\begin{rem} \label{c-null}
What also cannot be treated within this framework is the case where the function $\varepsilon$
is not away from $0$, in particular, if it is $0$ on a subset of positive boundary measure.
This would e.g.\ affect the case of inhomogeneous Neumann boundary conditions.
It is known that 
the resulting problem is of very different functional analytic quality and requires 
different methods, see \cite{Nit1}. 
\end{rem}

\section*{Acknowledgments} We wish to thank our colleagues K.~Gr\"oger (Berlin),
 H. Amann (Z\"urich), H.~Vogt (Clausthal), R.~Nittka (Leipzig) and P.~C.~Kunstmann (Karlsruhe) for 
valuable discussions on the subject of the paper.
We wish to thank the referee for his critical comments.

\end{document}